\documentclass[12pt]{article}
\usepackage{amssymb}
\usepackage{amsmath,amsfonts,amsthm,algorithm}

\usepackage{lineno}

\usepackage{graphicx}
\usepackage{subfig}

\usepackage{amssymb,epstopdf,subeqnarray,enumerate}
\usepackage{algpseudocode}
\usepackage{float}
\usepackage[toc,page]{appendix}
\usepackage[numbers,sort]{natbib}

\newtheorem{theorem}{Theorem}[section]

\newtheorem{remark}[theorem]{Remark}
\newtheorem{lemma}[theorem]{Lemma}
\newcommand{\R}{\mathbb{R}}

\newcommand{\de}{\,\text{d}}

\newcommand{\neQ}{\boldsymbol{Q}}
\newcommand{\neR}{\boldsymbol{R}}

\title{A boundary integral algorithm for the Laplace Dirichlet-Neumann
  mixed eigenvalue problem }

\author{Eldar Akhmetgaliyev$^*$, Oscar Bruno\thanks{Computing \& Mathematical Sciences, California Institute of Technology}, Nilima Nigam \thanks{Department of Mathematics, Simon Fraser University}}

\begin{document}
\maketitle

\begin{abstract}
  We present a novel integral-equation algorithm for evaluation of
  {\it Zaremba eigenvalues and eigenfunctions}, that is, eigenvalues
  and eigenfunctions of the Laplace operator with mixed
  Dirichlet-Neumann boundary conditions; of course, (slight
  modifications of) our algorithms are also applicable to the pure
  Dirichlet and Neumann eigenproblems. Expressing the eigenfunctions
  by means of an ansatz based on the single layer boundary operator,
  the Zaremba eigenproblem is transformed into a nonlinear equation
  for the eigenvalue $\mu$.  For smooth domains the singular structure
  at Dirichlet-Neumann junctions is incorporated as part of our
  corresponding numerical algorithm---which otherwise relies on use of
  the cosine change of variables, trigonometric polynomials and, to
  avoid the Gibbs phenomenon that would arise from the solution
  singularities, the Fourier Continuation method (FC). The resulting
  numerical algorithm converges with high order accuracy without
  recourse to use of meshes finer than those resulting from the cosine
  transformation. For non-smooth (Lipschitz) domains, in turn, an
  alternative algorithm is presented which achieves high-order
  accuracy on the basis of graded meshes. In either case, smooth or
  Lipschitz boundary, eigenvalues are evaluated by searching for zero
  minimal singular values of a suitably stabilized discrete version of
  the single layer operator mentioned above. (The stabilization
  technique is used to enable robust non-local zero searches.) The
  resulting methods, which are fast and highly accurate for high- and
  low-frequencies alike, can solve extremely challenging
  two-dimensional Dirichlet, Neumann and Zaremba eigenproblems with
  high accuracies in short computing times---enabling, in particular,
  evaluation of thousands of eigenvalues and corresponding
  eigenfunctions for a given smooth or non-smooth geometry with nearly
  full double-precision accuracy.
\end{abstract}

\noindent{\bf Keywords}


Fourier continuation, boundary integral
operators, mixed boundary conditions, Zaremba eigenvalue problem,
Laplace eigenvalue problem.


\pagestyle{plain}



\section{Introduction\label{Introduction}}
This paper presents a novel boundary integral strategy for the  {\em
  numerical solution of the Zaremba eigenproblem}, that is, numerical
approximation of eigenvalues $\lambda_j, j=1,2,...$ and associated
eigenfunctions $u_j \in H^1(\Omega)$ of the Laplace operator under
mixed Dirichlet-Neumann boundary conditions (cf. equation~\eqref{modelA} below); naturally, the main elements of our algorithms are also applicable to the pure Dirichlet and Neumann eigenproblems.

The use of boundary integral equations for the solution of Laplace
eigenproblems has been explored in a number of contributions,
including methods based on collocation~\cite{chenmultiple,kamiya} and
Galerkin~\cite{steinbachunger,steinbachDirichlet} boundary element
approaches for the Dirichlet and Neumann problems. The boundary
element strategy for three-dimensional Dirichlet eigenproblems
presented in~\cite{steinbachDirichlet,steinbachunger}, for example,
yields errors that decrease cubically with the spatial
mesh-sizes. However, as mentioned in~\cite{steinbachDirichlet}, ``the
convergence regions for the eigenvalues are still local'' and ``other
techniques have to be considered and analyzed in order to increase the
robustness''. Focusing on two-dimensional Laplace eigenvalue problems,
in this paper we present a Nystr\"om algorithm that can achieve any
user-prescribed order of convergence for smooth and non-smooth domains
alike, as well as a novel, robust, search algorithm that yields fast
eigenvalue convergence from nonlocal initial guesses---see
Section~\ref{Section:EVP} for details. To the best of our knowledge,
further, the present algorithm is the first boundary-integral method
for eigenvalue problems of Zaremba type.

Integral equation formulations  for eigenvalue problems are
advantageous as they 1)~Result in a
reduction in the problem dimensionality; and, as described in
section~\ref{Sectionbackground}, they 2)~Greatly facilitate efficient treatment of the
eigenfunction singularities that occur around corners and
Dirichlet-Neumann transition points. As a counterpart, however, the
integral form of the eigenvalue problem (cf. equation~\eqref{modelB} below) is nonlinear
(since the eigenvalue appears as part of the integral kernel), and
eigenvalues and eigenfunctions must therefore be found by means of an
appropriate nonlinear equation solver.

It is important to note that the eigenfunctions in
equation~\eqref{modelA} as well as the corresponding densities $\psi$
in~\eqref{modelC} exhibit singularities at corners and
Dirichlet-Neumann junctions.  In particular, in contrast to the
situation for the pure Dirichlet or Neumann eigenfunctions, even for a
smooth boundary $\Gamma$ the eigenfunctions of~\eqref{modelA} are
singular: they are elements of $H^1(\Omega)$ but not of
$H^2(\Omega)$. The specific asymptotic forms of these singularities
for both smooth and Lipschitz domains are described in
Section~\ref{Sectionbackground}.

In Section~\ref{Section:Method} we discuss novel discretization
strategies for our integral formulation of the Zaremba eigenvalue
problem which yield high order accuracy in spite of the poor
regularity of eigenfunctions and densities near Dirichlet-Neumann
junctions. In the smooth domain case our Zaremba eigensolver includes
an adaptation of the novel Fourier Continuation (FC)
method~\cite{bruno2009_1,bruno2009_2,albin2011} (which accurately
expresses non-periodic functions in terms of Fourier series; see
Section \ref{sec:smooth_geometries}) and it explicitly incorporates
the asymptotic behavior of solutions near the Dirichlet-Neumann
junction. For possibly non-smooth curves $\Gamma$, on the other hand,
an approach is introduced in Section \ref{sec:corner_geometries}
which, on the basis of a graded-mesh
discretizations~\cite{sag1964numerical,KUSSMAUL:1969,MARTENSEN:1963,KRESS:1990,COLTON:1998}, yields once again
high-order accuracy--- albeit not as efficiently, for smooth domains,
as that resulting from the FC-based algorithm
(cf. Remark~\ref{remark:cancellations_corner0} and
Section~\ref{sec:experiment1}).

A method for solution of Dirichlet eigenproblems for the Laplace
operator that, like ours, is based on detection of parameter values
for which a certain matrix is not invertible, was introduced
in~\cite{Fox1967}.  In that early contribution this {\em Method of
  Particular Solutions} (MPS) (which approximates eigenfunctions as
linear combinations of Fourier-Bessel functions) performs the
singularity search via a corresponding search for zeroes of the matrix
determinant. Subsequently,~\cite{molerreport} substituted this
strategy by a search for zeroes of minimum singular values---an idea
which, with some variations, is incorporated as part of the algorithm
proposed presently as well. A modified version of the MPS, which was
introduced in reference~\cite{trefethen2005}, alleviates some
difficulties associated with the conditioning of the method.

As it happens, however, a direct evaluation of the zeroes of the
smallest singular value $\eta_n(\mu)$ of our $n\times n$ discretized
boundary integral operator is highly challenging. Indeed, as shown in
Section~\ref{Section:EVP}, the function $\eta_n(\mu)$ is essentially
constant away from its roots, and therefore descent-based approaches
such as the Newton method fail to converge to the roots of $\eta_n$
unless an extremely fine mesh of initial guesses is used.  A modified
integral equation formulation (with associated smallest singular
values $\widetilde{\eta}_n(\mu)$) is introduced in
Section~\ref{Section:EVP} that, on the basis of ideas introduced
in~\cite{trefethen2005}, successfully tackles this difficulty
(cf. Remark~\ref{tref_betcke}).  As demonstrated in
Section~\ref{sec:numerical_results}, the resulting eigensolvers, which
are fast and highly accurate for high- and low-frequencies alike, can
solve extremely challenging two-dimensional Dirichlet, Neumann and
Zaremba eigenproblems with high accuracies in short computing
times. In particular, as illustrated in Section~\ref{sec:experiment6},
the proposed algorithms can evaluate thousands of Zaremba, Dirichlet
or Neumann eigenvalues and eigenfunctions with nearly full
double-precision accuracy for both smooth and non-smooth domains.
The algorithms presented in this paper can further be generalized to enable evaluation of eigenvalues of multiply-connected domains---for which integral eigensolvers can give rise to spurious resonances~\cite{chen2001boundary,chen2003spurious}. A preview of the capabilities of the generalized method for multiply connected domains is provided in Section~\ref{sec:experiment7}.

The recent contribution~\cite{zhao2014robust} relies on determination
of zeroes of matrix determinants to address, in the the context of the
pure Dirichlet eigenvalue problem, certain challenges posed by search
methods based on use of smallest singular values---which are generally
non-smooth function of $\mu$. As indicated in
Remark~\ref{remark:search_method}, however, a relatively
straightforward sign-changing procedure we use yields singular values
that vary smoothly (indeed, analytically!) with $\mu$, and thus
eliminates difficulties arising from non-smoothness. Note that
generalizations of the present methods to algorithms that rely on
iterative singular-value computations and fast evaluations of the
relevant integral
operators~\cite{bruno2001fast,rokhlin1993diagonal,bleszynski1996aim}
(which should enable solution of higher frequency/three-dimensional
problems) can be envisioned.

This paper is organized as follows: Section~\ref{ProblemSetup} describes the Laplace-Zaremba eigenvalue problem for a class of domains in $\mathbb{R}^2$ and Section~\ref{sec:integral_formulation} puts forth an equivalent boundary integral formulation based on representation of eigenfunctions via single layer potentials. Section~\ref{Sectionbackground} then discusses the singular structure of eigenfunctions and associated integral densities at both smooth and non-smooth Dirichlet-Neumann junctions; these results are incorporated in the high-order numerical quadratures described in Section~\ref{Section:Method}. Section~\ref{Section:EVP} introduces a certain normalization procedure which leads to an efficient eigenvalue-search algorithm. Once eigenvalues and corresponding integral densities have been obtained, the eigenfunctions can be produced with high-order accuracy throughout the spatial domain (including near boundaries) by means of a methodology presented in Section~\ref{Section:eigenfunctions}.  Section~\ref{sec:numerical_results}, finally, demonstrates the accuracy and efficiency of the eigensolvers introduced in this paper with a variety of numerical results.

\section{Preliminaries \label{ProblemSetup}}

We consider the eigenvalue problem
\begin{subeqnarray}\label{modelA}
\slabel{modelA_eq}  -\Delta u  &=& \lambda u,\qquad x\in \Omega\\
\slabel{modelA_bc}  u &=& 0, \qquad x \in \Gamma_D\\
\slabel{modelA_bc1}  \frac{\partial u}{\partial \nu} &=&0, \qquad x\in
  \Gamma_N.
\end{subeqnarray} 
Here $\Omega \subset \mathbb{R}^2$ denotes a bounded simply-connected domain with a Lipschitz boundary $\Gamma
= \partial \Omega$ and the Dirichlet and Neumann boundary portions
$\Gamma_D$ and $\Gamma_N$ are disjoint subsets of $\Gamma$.  Throughout this paper a curve in $\mathbb{R}^2$ said to be ``smooth'' (resp. analytic) if it admits a $C^{\infty}$ (resp. analytic) invertible parametrization. Similarly, infinitely differentiable functions of real variable are called ``smooth'' functions.

Let the piecewise-smooth boundary $\Gamma$ be expressed in the form
\begin{eqnarray}
\Gamma = \bigcup_{q=1}^{Q_N+Q_D} \Gamma_q ,
\end{eqnarray}
where $Q_D$ and $Q_N$ denote  the numbers of smooth  Dirichlet and Neumann boundary portions, and where for $1\leq q \leq Q_D$ (resp. $Q_D+1\leq q \leq Q_D+Q_N)$) $\Gamma_q$ denotes a {\em smooth} Dirichlet (resp. Neumann) segment of the boundary curve $\Gamma$. Clearly, letting
\begin{equation*}
J_D = \{ 1,\dots,Q_D\} \quad \mbox{and} \quad J_N = \{ Q_D+1,\dots,Q_D+ Q_N\}
\end{equation*}
we have that 
\begin{equation*}
\Gamma_D= \bigcup_{q \in J_D} \Gamma_q \quad  \mbox{and}  \quad \Gamma_N= \bigcup_{q \in J_N} \Gamma_q
\end{equation*}
are the (piecewise smooth) portions of $\Gamma$ upon which Dirichlet and Neumann boundary conditions are enforced, respectively.
Note that in view of the assumption above both  Dirichlet-Neumann junctions and non-smooth points in $\Gamma$
necessarily occur at a common endpoint of two segments $\Gamma_{q_1}$, $\Gamma_{q_2}$ ($1\leq q_1 , q_2\leq Q_D+Q_N$). Note, additionally, that consecutive values of the index $q$ do not necessarily correspond to consecutive boundary segments (see e.g. Figure~\ref{fig:boundary_decomposition}).   
\begin{figure}[H]
\centering
  \includegraphics[width=2in]{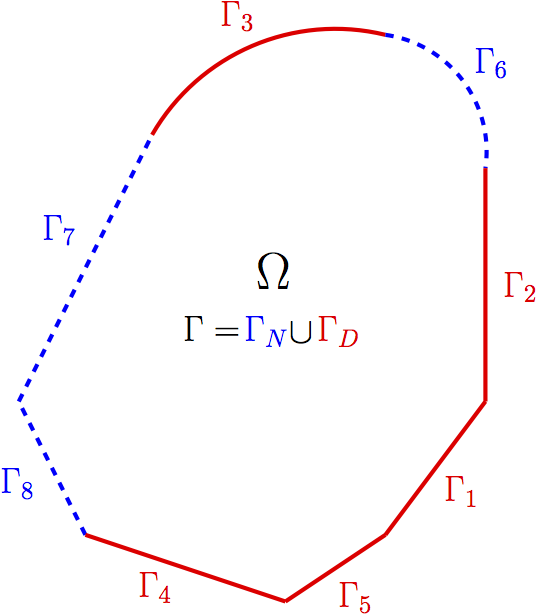}
  \caption{Boundary decomposition illustration. Dashed line: Neumann
    boundary. Solid line: Dirichlet boundary.}
  \label{fig:boundary_decomposition}
\end{figure}
\begin{remark}
\label{remark:smooth_boundary_portions}
Throughout this paper the decomposition of the curve $\Gamma$ is taken in such a way that no Dirichlet-Dirichlet or Neumann-Neumann junctions occur at a point at which the curve $\Gamma$ is smooth. In other words, every endpoint of $\Gamma_q$ is either a Dirichlet-Neumann junction or a non-smooth point of $\Gamma$. Clearly this is not a restriction:   two Dirichlet (resp. Neumann) segments $\Gamma_{q_1}$ and $\Gamma_{q_2}$ that meet at a  point at which $\Gamma$ is smooth can be combined into a single Dirichlet (resp. Neumann) segment.
\end{remark}

Following~\cite{CakoniColton2005}, for a given relatively open subset
$S\subseteq \Gamma$ we define the space
\begin{equation*}
H^{1/2}(S) = \{ \left. u \right| _{S} : u \in H^{1/2}(\Gamma) \},\qquad 
\widetilde{H}^{1/2}(S) = \{ \left. u \right| _{S} : \text{supp } u \subseteq S, u \in H^{1/2}(\Gamma) \}.
\end{equation*}
The dual of $ \widetilde{H}^{1/2}(S)$ is denoted by
\begin{equation*}
  H^{-1/2}(S) = \left(\widetilde{H}^{1/2}(S)\right)'.
\end{equation*}

\section{Integral formulation of the eigenvalue problem\label{sec:integral_formulation}} 
Introducing the Helmholtz Green function $G_\mu(x,y):= \frac{i}{4}
H^1_0(\mu|x-y|)$ and the associated single-layer potential
\begin{equation}\label{ansatz} 
  u(x):= \int_{\Gamma} G_\mu (x,y) \psi(y) \, ds_y \qquad (x \in \Omega)
\end{equation}
with surface density $\psi$, and relying on well known
expressions~\cite{COLTON:1983} for the values of the single layer $u$
and its normal derivative $\displaystyle\frac{\partial u}{\partial n}$
on $\Gamma$, we define the operators $\mathcal{A}^{(1)}:
H^{-1/2}(\Gamma) \to H^{1/2}(\Gamma_D)$ and $\mathcal{A}^{(2)}:
H^{-1/2}(\Gamma) \to H^{-1/2}(\Gamma_N)$ by
\begin{subeqnarray}
\label{modelB}
\slabel{modelB1}
\mathcal{A}^{(1)}_\mu[ \psi](x)  &=&  \displaystyle \int _{\Gamma}G_\mu(x,y)\psi (y)ds_y \qquad \mbox{for } x \in \Gamma_D, \\
\slabel{modelB2}
\mathcal{A}^{(2)}_\mu[ \psi](x)& =&  \displaystyle -\frac{\psi (x)}{2}+\int _{\Gamma}\frac{\partial}{\partial
    n_x}G_\mu(x,y)\psi (y)ds_y  \qquad \mbox{for } x \in \Gamma_N,
\end{subeqnarray}
and we then define 
\begin{equation}\label{modelC1}
\mathcal{A}_\mu =: H^{-1/2}(\Gamma)\to H^{1/2}(\Gamma_D)\times
H^{-1/2}(\Gamma_N)\qquad{\mbox by}\qquad \mathcal{A}_\mu [ \psi] =
(\mathcal{A}^{(1)}_\mu [ \psi] ,\mathcal{A}^{(2)}_\mu [ \psi] ) .
\end{equation}

The (linear) problem~\eqref{modelA} is
equivalent to the nonlinear problem of finding $\mu>0$ for which there holds:
\begin{equation}\label{modelC} 
\mbox{``The linear system $\mathcal{A}_\mu \psi  =0$ admits non-trivial solutions $\psi$''}. 
\end{equation}
To see this, let $u$ be given by equation~\eqref{ansatz}. Note that
$u$ does not vanish identically unless $\psi$ does---as can be
established by using uniqueness results for the Dirichlet exterior
problem and the jump relations satisfied by the single layer potential
and its normal derivative. Since, clearly, $- \Delta u = \mu^2 u$
throughout $\Omega$, further, it follows that for each $\mu$
satisfying~\eqref{modelC} the real number
\begin{equation*}
\lambda = \mu^2
\end{equation*}
is an eigenvalue of~\eqref{modelA}. Further, as established
in~\cite{AkhBruno2013} (cf. also~\cite{1997partial} for corresponding
results for the pure Dirichlet problem), every eigenvalue $\lambda$
equals $\mu^2$ for some $\mu\in\mathbb{R}$ satisfying~\eqref{modelC},
and the solutions $\psi$ of \eqref{modelC} are related to the
corresponding eigenfunctions $u$ of \eqref{modelA} via the
relation~\eqref{ansatz}. It follows that, as claimed, the eigenvalue
problem~\eqref{modelA} and problem~\eqref{modelC} are equivalent.
\begin{remark}\label{poles}
  As is known~\cite{1997partial}, complex values of $\mu$ do exist for
  which the integral form of the eigenvalue problem admits non-trivial
  solutions---although, they do not correspond to eigenvalues of the
  Laplace operator in the bounded domain $\Omega$. These values of
  $\mu$ do correspond to complex eigenvalues $\mu^2$ (also called
  ``scattering poles'') of the Laplace operator: they satisfy the
  Laplace eigenvalue equation {\em outside} $\Omega$ along with
  certain radiation conditions at infinity which allow for growth. The
  determination and study of these scattering poles, which is
  interesting in its own right~\cite{lenoir,moiseyev}, does not fall
  within the scope of this paper. A numerical method for evaluation of
  such poles for the Dirichlet exterior problem can be found
  in~\cite{steinbachunger}. In fact, we suggest that the stabilization
  strategy proposed in Section~\ref{Section:EVP} should be useful in
  the context of~\cite{steinbachunger} as well.
\end{remark}

Upon discretization of the problem \eqref{modelC}
(Section~\ref{Section:discretization_matrix}) we are lead to the
nonlinear problem of locating $\mu \in \mathbb{R}$ and ${\tt c} \in
\mathbb{R}^N$ which satisfy a discrete linear system of equations of
the form
\begin{equation}
\label{modelD} 
{\tt A}_{\mu} {\tt c} ={\tt 0}.
\end{equation} 
This problem is tackled in Section~\ref{Section:EVP} by consideration of the minimum singular value
$\widetilde{\eta}_n(\mu)$ (and corresponding right singular vector) of a certain augmented linear system related to~\eqref{modelD}: the quantities $\mu$ and ${\tt c}$ that satisfy~\eqref{modelC} are
obtained, simply, as a zero of the function $\sigma = \widetilde{\eta}_n(\mu)$ and the corresponding singular vector. The vector ${\tt c}$ provides a discrete approximation for the unknown density $\psi$; the eigenfunction $u$ itself
can then be obtained by means of a corresponding discrete version of the representation formula~\eqref{ansatz}.

\section{Singularities in eigenfunctions and integral equation   densities\label{Sectionbackground}} 

\begin{figure}[h!]
\centering
  \includegraphics[width=3in]{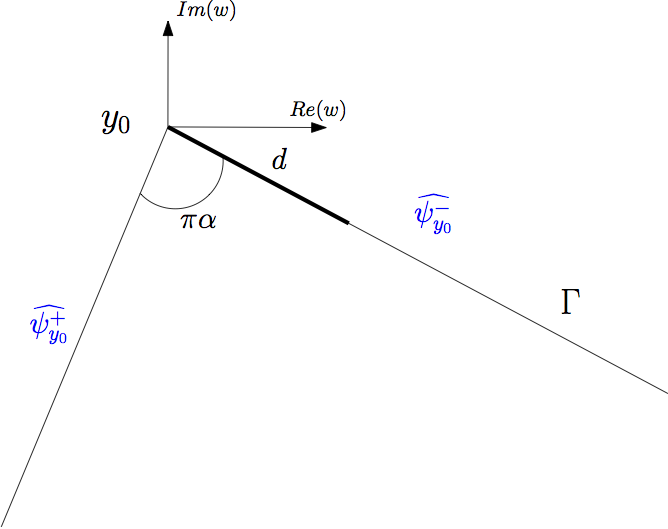}
  \caption{Point $y_0$ of singularity of the density function
    $\psi$. For $\alpha=1$ $y_0$ may or may not be a point at which
    $\Gamma$ is smooth (infinitely differentiable);
    cf. Remark~\ref{remark:smooth_boundary_portions}.}
  \label{fig:singularity}
\end{figure}
This section collects known results about the smoothness properties and singularities of the eigenfunctions in equation~\eqref{modelA} and the corresponding integral densities in equation~\eqref{modelB}. The singular character of these functions is incorporated as part of the discretization strategies we introduce in Section~\ref{Section:Method}.

Let $y_0 = (y^0_1,y^0_2)\in\Gamma$ be either a corner point (with associated corner angle $\alpha\pi $) at which a Dirichlet-Neumann junction may or may not occur, or a point around which the curve $\Gamma$ is smooth ($\alpha =1$) and which separates Dirichlet and Neumann regions within $\Gamma$. In either case $y_0$ is a singular point for the problem. Following the notations in~\cite{wigley1970mixed,AkhBruno2013}, in order to express the singular character of the eigenfunctions $u(y)$ ($y = (y_1,y_2) \in \Omega$) and corresponding integral equation densities $\psi(y)$ ($y = (y_1,y_2) \in \Gamma$) around $y_0$  we use certain functions $\widehat{u}_{y_0} = \widehat{u}_{y_0}(w)$, $\widehat{\psi}_{y_0}^+ = \widehat{\psi}_{y_0}^+(d)$ and $\widehat{\psi}_{y_0}^- = \widehat{\psi}_{y_0}^-(d)$. Here the left (resp. right) function $\widehat{\psi}_{y_0}^-$ (resp $\widehat{\psi}_{y_0}^+$) is the density as a function of the distance $d$ to the point $y_0$ in a small one-sided neighborhood immediately before (resp. immediately after) the point $y_0$ as the curve is traversed in the counterclockwise direction, and $w =(y_1-y^0_1) + i (y_2-y^0_2)$ is a complex variable (see Figure~\ref{fig:singularity}). The functions $\widehat{u}_{y_0}$, $\widehat{\psi}_{y_0}^+$ and $\widehat{\psi}_{y_0}^-$ are given by
\begin{equation}
\label{eq:real_to_complex}
\begin{split}
  &\widehat{u}_{y_0}(w) = u(y),  \\
  &\psi(y) =\widehat{\psi}_{y_0}^+(d(y)) \quad y \in \Gamma_{q1},\\
  &\psi(y) =\widehat{\psi}_{y_0}^- (d(y)) \quad y \in \Gamma_{q2}, 
\end{split}
\end{equation}
where, as mentioned above
\begin{equation}\label{local}
w = (y_1-y^0_1) + i (y_2-y^0_2) \quad ;\quad d(y) = \sqrt{(y_1-y^0_1)^2+(y_2-y^0_2)^2}.
\end{equation}

It is known~\cite{wigley1964,wigley1970mixed} that, under our
assumption that the curve $\Gamma$ is piecewise smooth, for any given
integer $\mathcal{N}$ and any given positive number $\varepsilon$ the
eigenfunctions in equation~\eqref{modelA} can be expressed in the form
\begin{equation}
\label{eigenfunction_asymptotics}
  \widehat{u}_{y_0}= \log(w) P_{y_0}^1 + \log(\bar{w})P_{y_0}^2 + P_{y_0}^3 + \mathcal{O}(w^{\mathcal{N}-\varepsilon})
\end{equation} 
for all $w$ in a neighborhood of the point, where
$P_{y_0}^1,P_{y_0}^2$ and $P_{y_0}^3$ are polynomials in $w$,
$\bar{w}$, $w^{1/(2 \alpha)}$, $\bar{w}^{1/(2 \alpha)}$ if $\alpha$ is
irrational; $P_{y_0}^1,P_{y_0}^2$ and $P_{y_0}^3$ are polynomials in
$w$, $\bar{w}$, $w^{1/(2 \alpha)}$, $\bar{w}^{1/(2 \alpha)}$,
$w^q\log(w)$, $\bar{w}^q \log(\bar{w})$ if $\alpha=p/q$ for some
integer $(p,q)=1$ and $q$ is odd, and $P_{y_0}^1,P_{y_0}^2$ and
$P_{y_0}^3$ are polynomials in $w$,$\bar{w}$,$w^{1/(2
  \alpha)}$,$\bar{w}^{1/(2 \alpha)}$,$w^{q/2} \log(w)$, $\bar{w}^{q/2}
\log(\bar{w})$ if $\alpha=p/q$ for some integer $(p,q)=1$ and $q$ is
even. (In fact, for Dirichlet-Dirichlet and Neumann-Neumann corners
some of the coefficients in the asymptotic expressions above vanish
and weaker singularities---polynomials in powers of $1/\alpha$ instead
of $1/(2 \alpha)$ in equation~\eqref{eigenfunction_asymptotics}---thus
result; see~\cite{wigley1964} for details. This point is not of any
practical significance in the context of this paper, however.) We
point out that in case the curve $\Gamma$ is smooth around $y_0$
($\alpha = 1$) and, thus, in view of
Remark~\ref{remark:smooth_boundary_portions} a Dirichlet-Neumann
junction exists at a smooth point $y_0$, the logarithmic terms
mentioned above actually drop out. Indeed, as shown in~\cite[Theorem
2.7]{AkhBruno2013}, in this case for any given $\mathcal{N}$ and
$\varepsilon >0$ the singularities of the eigenfunctions in
equations~\eqref{modelA} are characterized by expression of the form
\begin{equation}
\widehat{u}_{y_0}= P_{y_0} + \mathcal{O}(w^{\mathcal{N}-\varepsilon}) \label{asymptoticsSolution},
\end{equation}
where $P_{y_0}$ is a polynomial in $w$, $\bar{w}$, $w^{1/2}$  and $\bar{w}^{1/2}$.

The singular character of the density $\psi$ plays a fundamental role
in our proposed numerical strategy for discretization of the system of
integral equations~\eqref{modelB}.  To determine the singularities of the function $\psi$ we let $u_e$ denote the solution of an auxiliary Dirichlet
problem outside $\Omega$ with Dirichlet boundary values given by the
boundary values of the eigenfunction $u$:
\begin{equation*}
\begin{split}
  \Delta  u_e+\mu^2 u_e&=0 \text{ in } \Omega^c, \\
  u_e|_{\Gamma}&=u|_{\Gamma}. 
\end{split}
\end{equation*} 
In \cite[Proposition 1.1]{AkhBruno2013} it is shown that any given
eigenfunction $u$ satisfying~\eqref{modelA} can be expressed as a
single layer potential with a uniquely determined density $\psi$ given
by
\begin{equation}
\label{eq:phi_asymptotics}
\displaystyle \psi = \left .\frac{\partial u_e}{\partial n}\right\vert_{\Gamma} - \left .\frac{\partial u}{\partial n}\right\vert_{\Gamma}.
\end{equation}
In view of the regularity
result~\cite[Theorem 2.8]{AkhBruno2013} it follows that around  the point $(y^0_1,y^0_2)$ the functions $\widehat{\psi}_{y_0}^+$ and $\widehat{\psi}_{y_0}^-$ of equation~\eqref{eq:real_to_complex} are given in terms of the distance function~\eqref{local} by
\begin{equation}
\label{asymptoticsDensitycorner}
\begin{split}
\widehat{\psi}_{y_0}^+(d)=d^{1/(2 \alpha)-1} Q_{y_0}^1(d,d^{1/(2 \alpha)}, \log(d)) + \mathcal{O}(d^{\mathcal{N}-1-\varepsilon}),  \\
\widehat{\psi}_{y_0}^-(d)=d^{1/(2 \alpha)-1} Q_{y_0}^2(d,d^{1/(2\alpha)}, \log(d)) + \mathcal{O}(d^{\mathcal{N}-1-\varepsilon}) 
\end{split}
\end{equation}
for all $\mathcal{N}\in\mathbb{N}$. If $\Gamma$ is smooth at
$(y^0_1,y^0_2)$, in turn, the asymptotics of $\widehat{\psi}_{y_0}^+$
and $\widehat{\psi}_{y_0}^-$ around this point are given by
\begin{equation}
\label{asymptoticsDensity}
\begin{split}
\widehat{\psi}_{y_0}^+(d)=d^{-1/2} Q_{y_0}^3(d,d^{1/2}) + \mathcal{O}(d^{\mathcal{N}-1-\varepsilon}),  \\
\widehat{\psi}_{y_0}^-(d)=d^{-1/2} Q_{y_0}^4(d,d^{1/2}) + \mathcal{O}(d^{\mathcal{N}-1-\varepsilon}) 
\end{split}
\end{equation}
for any $\mathcal{N}\in\mathbb{N}$. Here $Q_{y_0}^i$, $i=1,4$ are
polynomials in the listed arguments. (Note that, while correct, the
boundary expressions~\eqref{asymptoticsDensitycorner} are less
detailed than the corresponding volumetric
expression~\eqref{eigenfunction_asymptotics}. In our context the
additional detail provided by
equation~\eqref{eigenfunction_asymptotics}, which shows that the
logarithmic terms are always accompanied by a $w^r$ factor for an
integer $r$ equal to either $q$ or $q/2$, do not carry any particular
significance.)

\section{High-order quadratures for integral eigensolvers \label{Section:Method}}
This section introduces high-order quadrature rules for the integral
operators in equation~\eqref{modelB}.  For domains with corners a
numerical integration method based on polynomial changes of variables
presented in Section~\ref{sec:corner_geometries}
(cf~\cite{sag1964numerical,KUSSMAUL:1969,MARTENSEN:1963,KRESS:1990,COLTON:1998})
ensures high-order integration in spite of the corner singularities.
For smooth boundaries $\Gamma$, further, a certain ``Fourier
Continuation'' algorithm is used to take advantage of the smoothness
of the domain boundary and thus yield even higher accuracies for the
singular Dirichlet-Neumann densities, eigenvalues and
eigenfunctions. The latter technique is described in
Section~\ref{sec:smooth_geometries}. Both of these descriptions rely
on expressions presented in Section~\ref{bound_dec} for the various
operators under consideration in terms of explicit parametrizations of
the boundary curve $\Gamma$.

\subsection{Parametrized  Operators\label{bound_dec}}

In view of the notations in Section~\ref{Introduction}, the operators~\eqref{modelB1},~\eqref{modelB2} applied to a density $\psi$ and evaluated at a given point $x
\in \Gamma$ can be expressed  in the form
\begin{equation} \label{eq:integrals_decomposed} 
\begin{split}
&\mathcal{A}^{(1)}_\mu[ \psi](x)  = \sum_{q=1}^{Q_D+Q_N} \int _{\Gamma_q} G_\mu(x,y)\psi (y)ds_y \qquad \mbox{for } x \in \Gamma_D,     \\
&\mathcal{A}^{(2)}_\mu[ \psi](x) =\displaystyle -\frac{\psi (x)}{2}+ \sum_{q=1}^{Q_D+Q_N} \int _{\Gamma_q} \frac{\partial}{\partial n_x} G_\mu(x,y)\psi (y)ds_y  \qquad \mbox{for } x \in \Gamma_N.  
\end{split}
\end{equation}
We seek expressions of these operators in terms of parametrizations of
the underlying integration curves. Without loss of generality, we
assume the boundary curve $\Gamma$ is parametrized by a single
piecewise-smooth vector function $y=z(\tau)=(z_1(\tau),z_2(\tau))$
($a\leq \tau < b$) satisfying $ (z_1')^2+ (z_2')^2> \delta$ for some
scalar $\delta >0$ at each point where $z$ is differentiable; the
 parametrization we use for integration on each one of the (smooth) Dirichlet
and Neumann segments $\Gamma_q$ is then taken to equal the relevant
restriction of the function $z$ to a certain interval $[a_q,b_q]$,
$a_q \leq b_q$. Clearly, $[a,b]=\displaystyle \cup_{q=1}^{Q_D+Q_N}
[a_q, b_q]$ and $ [a_{q_1}, b_{q_1}]\cap [a_{q_2}, b_{q_2}]$ is either
the empty set or a set containing a single point.

To evaluate each one of the integrals
in~\eqref{eq:integrals_decomposed} for a point $x\in \Gamma$ we rely
on the decomposition
\begin{equation*}
H^{(1)}_\nu(\zeta) = F_{\nu}^{(0)}(\zeta) \log(\zeta) + F_{\nu}^{(1)}(\zeta) ,
\end{equation*}
where $F_{\nu}^{(0)}$ and $F_{\nu}^{(1)}$ are analytic functions
(cf.~\cite[p. 68]{colton1984}). For each $q_2 \in J_D \cup J_N$ two
integrals over $\Gamma_{q_2}$ appear in
equation~\eqref{eq:integrals_decomposed}. Using the substitutions
$x=z(t)$ and $y=z(\tau)$ and assuming $x=z(t)\in\Gamma_{q_1}$ for a
certain $q_1\in J_D \cup J_N$ ($t\in [a_{q_1},b_{q_1}]$), we express
each one of the aforementioned integrals over $\Gamma_{q_2}$ in terms
of the operator
\begin{equation}
\mathcal{\widetilde{I}}_{q_1,q_2}[\widetilde{\varphi}](t) =   \int_{a_{q_2}}^{b_{q_2}}\left\{\widetilde{K^1}(t,\tau) \log R^2(t,\tau) +\widetilde{K^2}(t,\tau) \right\} \widetilde{\varphi}(\tau)\de \tau, \label{eq:parametrized_operators}
\end{equation}
where
\begin{equation}
\label{eq:psi_to_phi}
R(t,\tau) := |x-y| =|z(t)-z(\tau)|\quad, \quad \widetilde{\varphi}(\tau) =\psi (z(\tau)).
\end{equation}
Here the kernels $\widetilde{K^1}(t,\tau)$ and $\widetilde{K^2}(t, \tau)$ denote functions that depend on the evaluated operator: for the integrals included in the operator $\mathcal{A}^{(1)}_\mu[ \psi]$ these kernels are given by the products of the arc-length $\sqrt{(z'(t))^2}$ and the factors $F_{\nu}^{(0)}$ and $F_{\nu}^{(1)}$  for $\zeta = kR(t,\tau)$ and $\nu= 0$. For the integrals included in the operator $\mathcal{A}^{(2)}_\mu[ \psi]$, on the other hand, an additional smooth factor is included, and $\nu=1$ is taken; see~\cite[p. 68]{colton1984} for details. In particular, for each $t\in [a,b]$, $\widetilde{K^1}(t,\tau)$ and $\widetilde{K^2}(t, \tau)$ are smooth (resp. analytic) functions of $\tau$ for all $\tau\in[a_{q_2}, b_{q_2}]$ provided $y(\tau)$ is itself smooth (resp. analytic). (The notations $\widetilde{K^1}$, $\widetilde{K^2}$ and $ \widetilde{\varphi}$ are used in connection with the basic  parametrization $z$; corresponding kernels $K^1$, $K^2$ and density $\varphi$, which include additional ``smoothing'' reparametrizations, are utilized in Sections~\ref{sec:smooth_geometries} and~\ref{sec:corner_geometries} below.) 

\begin{remark}\label{remark:kernel_singularitites}
Clearly the kernel  in the integral
operator~\eqref{eq:parametrized_operators} (the quantity in curly brackets in this equation) is singular, smooth, or nearly singular depending, respectively, on whether 1.~$q_1=q_2 = q$ (that is, $t, \tau \in [a_{q}, b_{q}]$); 2.~$q_1\not=q_2$ and $t$ is ``far'' from
$[a_{q_2}, b_{q_2}]$, or 3.~$q_1\not=q_2$ and $t$ is ``close'' to
$[a_{q_2}, b_{q_2}]$. The significance of the terms ``far'' and ``close'' and corresponding selections of algorithmic thresholds is taken up in Remark~\ref{remark:cases}.
\end{remark}
\begin{remark}\label{remark:cases}
In the case $q_1 \not = q_2$ point $t$ is considered to be ``far'' from the interval $[a_{q_2},b_{q_2}]$ (case 2. in Remark~\ref{remark:kernel_singularitites})  provided
\begin{equation}\label{far_close_condition}
\min(|z_{q_1}(t) -z_{q_2}(a_{q_2})|,|z_{q_1}(t) -z_{q_2}(b_{q_2})|) > h_1 ,
\end{equation}
that is, provided the minimum euclidean distance between $z_{q_1}(t)$
and the interval endpoints larger than $h_1$, where $h_1$ is a given
(user-provided) parameter which is to be selected so as to maximize
overall accuracy.  Otherwise the point $t$ is considered to be
``close'' to the interval $[a_{q_2},b_{q_2}]$ (case 3. in
Remark~\ref{remark:kernel_singularitites}).
\end{remark}

\begin{remark}\label{remark:density_singularitites}
  In the case of a Lipschitz domain
  equations~\eqref{asymptoticsDensitycorner} imply that the asymptotic
  behavior of the integral density $\widetilde{\varphi}=
  \widetilde{\varphi}(\tau)$ near $\tau=a_{q_2}$ and near $\tau =
  b_{q_2}$ is characterized, respectively, by expressions of the form
\begin{equation}
\label{eq:phi_asymptotics_tau}
\begin{split}
\widetilde{\varphi}(\tau) &= (\tau - a_{q_2})^{1/(2\alpha)-1} P_1((\tau - a_{q_2})^{1/(2\alpha)}, (\tau - a_{q_2}), \log{(\tau - a_{q_2})}) + \mathcal{O}((\tau - a_{q_2})^{\mathcal{N}-1-\varepsilon}),\\
\widetilde{\varphi}(\tau) &= (\tau - b_{q_2})^{1/(2\alpha)-1} P_2((\tau - b_{q_2})^{1/(2\alpha)}, (\tau - b_{q_2}), \log{(\tau - b_{q_2})}) + \mathcal{O}((\tau - b_{q_2})^{\mathcal{N}-1-\varepsilon}),
\end{split}
\end{equation}
where,  for any given integer $\mathcal{N}$, $P_1$ and $P_2$ are polynomials---which, of course, depend on $q_2$ and $\mathcal{N}$.  In the case of smooth $\Gamma$, in turn, the equations in~\eqref{asymptoticsDensity} tell us that the the asymptotic behavior of $\widetilde{\varphi}$ near $\tau=a_{q_2}$ and near $\tau = b_{q_2}$ is given, respectively,  by the relations
\begin{equation}\label{eq:phi_asymptotics_tau_smooth}
\begin{split}
\widetilde{\varphi}(\tau) &= (\tau - a_{q_2})^{-1/2} P_3((\tau - a_{q_2})^{1/2}) + \mathcal{O}((\tau - a_{q_2})^{\mathcal{N}-1-\varepsilon}),\\
\widetilde{\varphi}(\tau) &= (\tau - b_{q_2})^{-1/2} P_4((\tau - b_{q_2})^{1/2}) + \mathcal{O}((\tau - b_{q_2})^{\mathcal{N}-1-\varepsilon}),
\end{split}
\end{equation}
where, once again, for any given integer $\mathcal{N}$, $P_3$ and $P_4$ are polynomials that depend on $q_2$ and $\mathcal{N}$. 
\end{remark}

We now turn to the design of high-order accurate quadrature rules for
integrals of the type~\eqref{eq:parametrized_operators} which, by
necessity, must take into account the singular character of the
integrand---including the explicit logarithmic singularities and near
singularities mentioned above as well as the singularities that the
(unknown) density function~$ \widetilde{\varphi}$ possesses at
Dirichlet-Neumann junctions and corner points
(Remarks~\ref{remark:kernel_singularitites}
and~\ref{remark:density_singularitites}). To do this we consider
separately the cases in which the overall curve $\Gamma$ is smooth
(Section~\ref{sec:smooth_geometries}) and Lipschitz
(Section~\ref{sec:corner_geometries}).

\subsection{FC-based algorithm for evaluation of the integral
  operators~\eqref{eq:parametrized_operators} \label{sec:smooth_geometries}}

This section concerns the mixed Dirichlet-Neumann boundary value
problem on a  domain $\Omega$ with a smooth boundary $\Gamma$; throughout this section we therefore consider integrals
of the form~\eqref{eq:parametrized_operators} where $y=z(\tau)$ is a
\emph{smooth} function on the entire interval $[a,b]$. As pointed out in Remark~\ref{remark:density_singularitites},  the global smoothness of $\Gamma$ ensures that the singularities of  the
unknown density function $\widetilde{\varphi}= \widetilde{\varphi}(\tau)$ are characterized by the expression~\eqref{eq:phi_asymptotics_tau_smooth} rather than~\eqref{eq:phi_asymptotics_tau}.

\subsubsection{FC-based algorithm: Cosine transformation and density regularization \label{sec:cos_transformation}}

As mentioned in Section~\ref{bound_dec}, all singularities must  be taken into account in order to obtain an overall high-order accurate solver. In what follows we describe an
approach that simultaneously eliminates the density singularities and accounts for both the logarithmically singular kernel $\widetilde{K^1}\cdot \log R^2$  and smooth kernel  $\widetilde{K^2}$ in~\eqref{eq:parametrized_operators} and  thereby results in a 
high-order accurate  method for evaluation of this integral operator. To do this we proceed by
introducing a cosine transformation for the integral in a segment $\Gamma_{q_2}$---after a necessary scaling to the
interval $[-1,1]$.

In detail we first map each parameter interval $[a_{q_2},b_{q_2}]$ to the interval $[-1,1]$ by means of the linear transformations
\begin{equation}
\label{linear_transformation_smooth_case}
\displaystyle \tau = \xi_{q_2}(\rho):=\frac{(b_{q_2}-a_{q_2}) \rho+ (a_{q_2}+b_{q_2})}{2}.
\end{equation}
Clearly, values of $t$ within $[a_{q_1},b_{q_1}]$ are given by $t = \xi_{q_1}(r)$ for some $r\in [-1,1]$. Denote
\begin{equation*}
\widetilde{K}_{q_1,q_2} (r,  \rho) =  \widetilde{K^1}(\xi_{q_1}(r),\xi_{q_2}(\rho)) \log R^2(\xi_{q_1}(r),\xi_{q_2}(\rho)) + \widetilde{K^2}(\xi_{q_1}(r),\xi_{q_2}(\rho)).
\end{equation*}
After application of this transformation, the integral~\eqref{eq:parametrized_operators} becomes
\begin{equation}
\label{eq:integral_after_linear_cov}
\begin{split}
\mathcal{\widetilde{I}}_{q_1,q_2}[\widetilde{\varphi}](r) =\frac{b_{q_2}-a_{q_2}}{2} &\int_{-1}^1 \widetilde{K}_{q_1,q_2} (r,  \rho) \widetilde{\varphi} (\xi_{q_2}(\rho))  \de\rho.
\end{split}
\end{equation}
Introducing the sinusoidal change of variables 
\begin{equation}\label{sinusoidal_cov}
r =  \cos(s) \qquad \mbox{and}\qquad  \rho =  \cos(\sigma), 
\end{equation}
and letting
\begin{equation}
\label{eq:phi_tilde}
\varphi_{q_2}(\sigma) = \widetilde{\varphi} (\xi_{q_2}( \cos(\sigma)))
\end{equation}
and 
\begin{equation}
\label{eq:smooth_case_kernel_transformation_s_sigma_one_kernel}
K_{q_1,q_2} (s,\sigma) =  \widetilde{K}_{q_1,q_2} (\cos(s),\cos(\sigma))
\end{equation}
expression~\eqref{eq:integral_after_linear_cov} becomes
\begin{equation}\label{eq:before_fc}
\mathcal{I}_{q_1,q_2}[\varphi_{q_2}](s) =\frac{b_{q_2}-a_{q_2}}{2}   \int_0^{\pi} K_{q_1,q_2} (s, \sigma) \varphi_{q_2}(\sigma) \sin(\sigma) d\sigma.
\end{equation}
\begin{lemma}
\label{smooth_in_theta}
Let $\Gamma$ be a smooth curve. Then the product $\varphi_{q_2}(\sigma)\sin (\sigma)$ is a smooth function of $\sigma$ for $\sigma \in [0,\pi]$. 
\end{lemma}
\begin{proof}
  In view of Remark~\ref{remark:smooth_boundary_portions} and the
  global smoothness of $\Gamma$ it follows that Dirichlet-Neumann
  junctions occur at both endpoints of $\Gamma_{q_2}$ and, thus, the
  asymptotic behavior of the density function $\widetilde{\phi}(\tau)$
  takes the form~\eqref{eq:phi_asymptotics_tau_smooth}. But, clearly,
  for any integer $\ell \geq -1$ for $\tau$ near $a_{q_2}$ (which
  corresponds to $\rho$ near $-1$ and $\sigma$ near $\pi$), up to
  multiplicative constants we have
\begin{equation*}
(\tau - a_{q_2})^{\ell/2}\sin(\sigma) \sim (\rho + 1)^{\ell/2}\sin(\sigma) \sim \sin(\sigma)^{\ell+1}.
\end{equation*}
Similarly  for $\tau$ near $b_{q_2}$ (which corresponds to   $\rho$ near $1$, to $\sigma$ near zero), once again up to multiplicative constants there holds
\begin{equation*}
(\tau - b_{q_2})^{\ell/2}\sin(\sigma) \sim (\rho - 1)^{\ell/2}\sin(\sigma) \sim \sin(\sigma)^{\ell+1}.
\end{equation*}
It then follows from equation~\eqref{asymptoticsDensity} that $\varphi_{q_2}(\sigma)\sin (\sigma)$ is a smooth function of $\sigma$ and the proof is complete.
\end{proof}

\subsubsection{FC-based algorithm: Fourier Continuation\label{subsection_FC}} 
We seek to produce high order quadrature rules for evaluation of the integral operator $\mathcal{I}_{q_1,q_2}[\varphi](r)$ in equation~\eqref{eq:before_fc} by exploiting existing explicit formulae for evaluation of integrals of the form
\begin{equation}
\label{log_ints}
\int_0^{\pi} \log|r  - \cos(\sigma)| \cos(n\sigma) d\sigma \quad \mbox{and}\quad
\int_0^{\pi} \log|r  - \cos(\sigma)| \sin(n\sigma) d\sigma
\end{equation}
for all real values of $r$ 
\begin{remark} \label{remark:log_integrals} Explicit expressions for
  the integrals~\eqref{log_ints} in the case of cosine integrands and
  $|r| \leq 1$ can be found
  in~\cite{bruno2007regularity,yan1988integral,masonchebyshev}. Corresponding
  expressions for the sine integrands and for the case $|r|>1$, which
  were introduced in~\cite{AkhBrunoReitich}, in turn, are reproduced
  in equations~\eqref{symms_operator}-\eqref{eq:Am_formula}
  below. Note that values $\left|r \right| \leq 1$ give rise to weakly
  singular logarithmic integration, while values $\left| r \right| >
  1$ result in smooth integrands which, however, are nearly singular
  for values of $r$ close to $1$ and $-1$.
\end{remark}
In order to take advantage of the expressions~\eqref{log_ints} we need
to express the integrand in equation~\eqref{eq:before_fc} in terms of
the functions $\cos(n \sigma)$ and $\sin(n \sigma)$; we do this by
relying on a certain Fourier Continuation
method~\cite{bruno2009_1,bruno2009_2,albin2011} which we discuss in
what follows.

To demonstrate the Fourier Continuation procedure as it applies in the
present context we consider the function $\widetilde{f}(\rho) =
\arccos(\rho)$ whose asymptotic expansions around $\rho = 1$ and $\rho
= -1$, just like those for the function
$\widetilde{\varphi}(\xi_{q_2}(\rho))$, contain the singular powers
$(\rho-1)^{n/2}$ and $(\rho+1)^{n/2}$, respectively, for all positive
odd values of the integer $n$. (Note in passing that the function
$\widetilde{\varphi}(\xi_{q_2}(\rho))$ contains, additionally, the
smooth terms $(\rho-1)^{n/2}$ and $(\rho+1)^{n/2}$ that result for
{\em even} values of $n$; this, however, is of no significance in the
present example.) The left portion of
Figure~\ref{fig:FC_demonstration} displays the function
$\widetilde{f}$ on the interval $[-1,1]$.  Under the cosine change of
variables used earlier in this section in the definition of the
function $\widetilde{\varphi}(\xi_{q_2}(\rho))$, this function becomes
$f(\sigma)=\widetilde{f}(\cos(\sigma)) = \sigma$ on the interval
$[0,\pi]$. The expansion sought above for the function
$\widetilde{f}$, would, in this simplified example, require
representation of the function $f(\sigma) = \sigma$ in a rapidly
convergent series in $\cos(n \sigma)$ and $\sin(n \sigma)$. This
objective could be achieved by means of an adequate globally smooth
and $2\pi$-periodic continuation of the function $f$. Although
theoretically this does not present difficulties, a fast and stable
numerical algorithm for evaluation of such a Fourier series has been
provided only recently---this is the Fourier Continuation (FC) method
mentioned above~\cite{bruno2009_1,bruno2009_2,albin2011}. A brief
overview in these regards is presented in appendix~\ref{sec:FC}. The
result of an application of the FC approach to the function
$f(\sigma)$ discussed above is given in
Figure~\ref{fig:FC_demonstration}: the desired globally smooth
periodic function, which is given as a rapidly convergent Fourier
expansion in terms of the functions $\cos(n \sigma)$ and $\sin(n
\sigma)$, is depicted on the right portion of this figure.
\begin{figure}[h!]
\centering
  \includegraphics[width=6in]{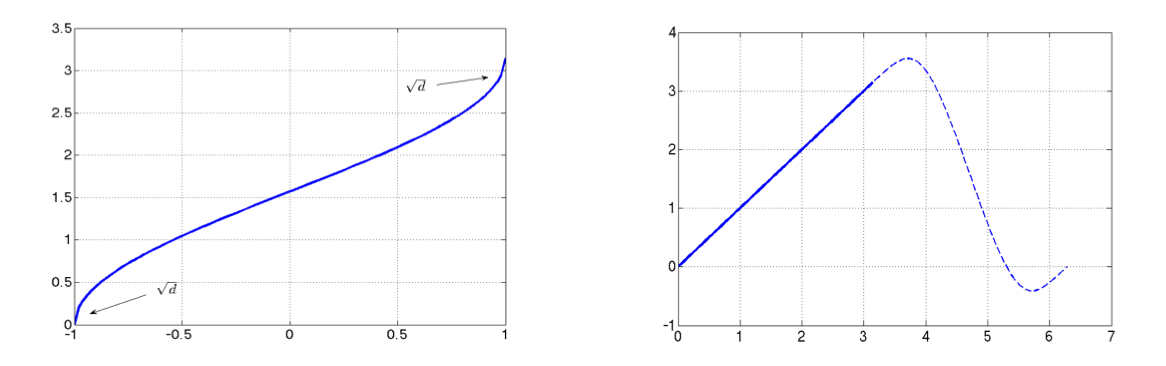}
  \caption{A cosine change of variables on the (singular) curve
    displayed in the left image produces the $y=x$ curve between $0$
    and $\pi$ in the right image. An application of the Fourier
    Continuation method then gives rise to the dashed-line
    continuation to a fully $2\pi$-periodic globally-smooth function
    shown on the right image.}
  \label{fig:FC_demonstration}
\end{figure}

\subsubsection{FC-based algorithm: Canonical kernel decomposition\label{ker_dec}}
This section provides canonical decompositions for the integral
kernels in equation~\eqref{eq:before_fc} in terms of smooth factors
and factors that explicitly display logarithmic singularities and
near-singularities. We consider three cases that parallel those in
Remark~\ref{remark:kernel_singularitites}; in each case the
decomposition depends on the singular character of the kernel
$K_{q_1,q_2}(s,\sigma)$:
\begin{enumerate}[1)]
\item{\bf Case $q_1\not=q_2$ and $t$ is ``far'' from $[a_{q_2},
    b_{q_2}]$ \label{point1}}

  The kernel $K_{q_1,q_2}(s,\sigma)$ is a smooth function of $\sigma$
  in this case (Remark~\ref{remark:kernel_singularitites}).

\item{\bf Case $q_1=q_2 = q$ \label{point2}} 

Introducing the kernels
\begin{equation}
\label{eq:smooth_case_kernel_transformation_same_q}
\begin{split}
  &K^1_{q,q}(s,\sigma) =  2 \widetilde{K^1}(\xi_{q}(  \cos(s)),\xi_{q}(\cos(\sigma) )), \\
  &K^2_{q,q}(s,\sigma)
  =\widetilde{K^2}(\xi_{q}(\cos(s)),\xi_{q}(\cos(\sigma))) +
  \frac{K^1_{q,q}(s,\sigma)}{2} \log\left(\frac{
      R^2(\xi_{q}(\cos(s)),\xi_{q}(\cos(\sigma)))}{|\cos(s) -
      \cos(\sigma)|^2}\right),
\end{split}
\end{equation}
(where, for $s= \sigma$, an appropriate limit as $\sigma \to s$ is
taken for the fraction in the argument of the logarithm in
equations~\eqref{eq:smooth_case_kernel_transformation_same_q} and where
the quantity $K^1_{q,q}$ used in the second equation is defined in the
first equation), the required decomposition is
\begin{equation}\label{decomposition_q1_equal_q2}
K_{q,q} (s, \sigma) = K^1_{q,q}(s, \sigma) \log\left|\cos(s)-\cos(\sigma)\right| + K^2_{q,q}(s, \sigma).
\end{equation}

\item{\bf Case $q_1\not=q_2$ and $t$ is ``close'' to $[a_{q_2}, b_{q_2}]$\label{point3}}

  As mentioned in Remark~\ref{remark:kernel_singularitites}, in this
  case the kernel is nearly singular. A specialized procedure is
  described in what follows which, using
  equation~\eqref{linear_transformation_smooth_case} beyond its domain
  of definition---for values of $\rho$ and $\tau$ for which $|\tau|
  >1$---gives rise to a useful decomposition in the present case.  In
  detail, taking advantage of the smoothness of the curve $\Gamma$
  (which is assumed throughout this section) we use the changes of
  variables~\eqref{linear_transformation_smooth_case}
  and~\eqref{sinusoidal_cov} that relate $\tau$ to $\sigma$ to also
  express $t$ as a function of $s$. We thus define a function
  $r^\mathrm{out}(s)$ by means of the relation
\begin{equation}\label{weird_function}
t=\xi_{q_2}(r^\mathrm{out}(s) ) = \xi_{q_1}( \cos(s) );
\end{equation}
it is easy to check that, for a given $s\in [0,\pi]$, $r = r^\mathrm{out}(s)$ lies in the interval
\begin{equation}\label{s_int}
\left[\frac{2 a_{q_1} -a_{q_2} - b_{q_2}}{b_{q_2}-a_{q_2}} , \frac{2 b_{q_1} -a_{q_2} - b_{q_2}}{b_{q_2}-a_{q_2}}\right],
\end{equation} 
and, in particular, $r^\mathrm{out}(s)$ is {\em outside} the interval
$[-1,1]$.  Owing to the continuity of the boundary parametrization
$z(t)$, further, $r^\mathrm{out}(s)$ is close to either $1$ or $-1$
for values of $s$ near $0$ or $\pi$.

On the basis of the sinusoidal change of variables $\rho = \cos(\sigma)$ (cf.~\eqref{sinusoidal_cov})  and the reparametrization $r = r^\mathrm{out}(s)$ we can now produce the desired decomposition for the kernel~\eqref{eq:smooth_case_kernel_transformation_s_sigma_one_kernel}: letting
\begin{equation}
\label{eq:smooth_case_kernel_transformation}
\begin{split}
K^1_{q_1,q_2}(s,\sigma) =  2 \widetilde{K^1}(&\xi_{q_1}(  r^\mathrm{out}(s)),\xi_{q_2}(\cos(\sigma) )) ,\\
K^2_{q_1,q_2}(s,\sigma)  =\widetilde{K^2}(&\xi_{q_1}(r^\mathrm{out}(s)),\xi_{q_2}(\cos(\sigma)))\   + \\
&\frac{K^1_{q_1,q_2}(s,\sigma)}{2} \log\left(\frac{ R^2(\xi_{q_1}(r^\mathrm{out}(s)),\xi_{q_2}(\cos(\sigma)))}{|r^\mathrm{out}(s) - \cos(\sigma)|^2}\right),
\end{split}
\end{equation}
(where the quantity $K^1_{q_1,q_2}$ used in the second equation is
defined in the first equation) we obtain
\begin{equation}\label{decomposition_q1_close_q2}
K_{q_1,q_2}(s, \sigma) = K^1_{q_1,q_2}(s, \sigma) \log\left|r^\mathrm{out}(s) - \cos(\sigma)\right| + K^2_{q_1,q_2}(s, \sigma).
\end{equation}

\end{enumerate}

\subsubsection{FC-based algorithm: Numerical integration \label{sec:high_order_quadratures}} 
This section describes numerical methods for evaluation of the
integrals $\mathcal{I}_{q_1,q_2}$ (equation~\eqref{eq:before_fc}) for
the three cases considered in Section~\ref{ker_dec}. In each case
$2\pi$-periodic Fourier continuation expansions of the form
\begin{equation}\label{phi_j_q_series}
\phi_{q_1,q_2}^j(s,\sigma) \sim \sum _{\ell=0}^{n} \alpha^{j}_{\ell} \cos(\ell \sigma) +  \beta^{j}_{\ell} \sin(\ell \sigma)\qquad j=1,2
\end{equation}
(that is, partial Fourier continuation expansions in the variable
$\sigma$ with coefficients $\alpha^{j}_{\ell} = \alpha^{j}_{\ell}(s)$
and $\beta^{j}_{\ell} = \beta^{j}_{\ell}(s)$) are used, where
$\phi_{q_1,q_2}^j=\phi_{q_1,q_2}^j(s,\sigma)$ are certain smooth
functions of $s$ and $\sigma$ for $0\leq s,\sigma\leq \pi$. With
reference to equation~\eqref{eq:before_fc},
Lemma~\ref{smooth_in_theta} and
equations~\eqref{decomposition_q1_equal_q2}
and~\eqref{decomposition_q1_close_q2} (and as detailed in what
follows), in all three cases $\phi_{q_1,q_2}^j$ denotes the product of
$\varphi_{q_2}(\sigma)\sin (\sigma)$ and the relevant smooth function
that multiplies a singular $\log$ (which we may call the ``$\log$
prefactor'') for $j=1$, and the product of $\varphi_{q_2}(\sigma)\sin
(\sigma)$ and the smooth remainder term for $j=2$. Note that in case
1) of section~\ref{ker_dec} the $\log$ prefactor vanishes.

The numerical quadrature methods for each of the three cases considered in Section~\ref{ker_dec} are given in what follows.
\begin{enumerate}[1)]
\item{\bf Case $q_1\not=q_2$ and $t$ is ``far'' from
$[a_{q_2}, b_{q_2}]$}

From point ~\ref{point1}) in Section~\ref{ker_dec}, in this case we set 
\begin{equation}\label{phi_j_q_1}
\phi_{q_1,q_2}^1(s,\sigma) = 0\quad,\quad \phi_{q_1,q_2}^2(s,\sigma) = K_{q_1,q_2}(s,\sigma) \varphi_{q_2}(\sigma) \sin(\sigma).
\end{equation}
The desired quadrature rule for~\eqref{eq:before_fc} results from use of~\eqref{phi_j_q_series} and explicit evaluation of the integrals of sines and cosines in the resulting approximate expression
\begin{equation}\label{eq:before_fc_repeated}
\mathcal{I}_{q_1,q_2}[\varphi_{q_2}](s) \sim \frac{b_{q_2}-a_{q_2}}{2}  \sum_{\ell=0}^{n}  \int_0^{\pi}  [\alpha^{2}_{\ell}(s) \cos(\ell \sigma) +  \beta^{2}_{\ell}(s) \sin(\ell \sigma)] d\sigma.
\end{equation}

\item{\bf Case $q_1=q_2 = q$}

Using the kernel  decomposition~\eqref{decomposition_q1_equal_q2} we set 
\begin{equation}\label{phi_j_q_2}
\phi^1_{q,q}(s,\sigma) = K^1_{q,q}(s,\sigma) \varphi_{q}(\sigma) \sin(\sigma)\quad\mbox{and}\quad \phi^2_{q,q}(s,\sigma) = K^2_{q,q}(s,\sigma) \varphi_{q}(\sigma) \sin(\sigma),
\end{equation}
so that in view of~\eqref{phi_j_q_series} we have
\begin{equation}
\label{eq:before_fc2}
\begin{split}
\mathcal{I}^{(q,q)}[\varphi_{q}](s) \sim \frac{b_{q}-a_{q}}{2}&  \sum _{\ell=0}^{n}  \int_0^{\pi}  \log|\cos(s) - \cos(\sigma)| [\alpha^{1}_{\ell}(s) \cos(\ell \sigma) +  \beta^{1}_{\ell}(s) \sin(\ell \sigma)] d\sigma  \\
+ \frac{b_{q}-a_{q}}{2}&\sum _{\ell=0}^{n}   \int_0^{\pi}  [\alpha^{2}_{\ell}(s) \cos(\ell \sigma) +  \beta^{2}_{\ell}(s) \sin(\ell \sigma)] d\sigma;
\end{split}
\end{equation}
Our quadrature rule for~\eqref{eq:before_fc} in the present case thus results from explicit evaluation of integrals of sines and cosines as well as integrals of the form~\eqref{log_ints} with $r = \cos{s}$ (equations~\eqref{symms_operator} and~\eqref{eq:Am_formula} below).
\item{\bf Case $q_1\not=q_2$ and $t$ is ``close'' to $[a_{q_2}, b_{q_2}]$}

Using the decomposition~\eqref{decomposition_q1_close_q2} and setting
\begin{equation}\label{phi_j_q_3}
\begin{split}
\phi_{q_1,q_2}^1(s,\sigma) = K^1_{q_1,q_2}(s,\sigma) \varphi_{q_2}(\sigma) \sin(\sigma) \quad \mbox{and} \quad \phi_{q_1,q_2}^2(s,\sigma) = K^2_{q_1,q_2}(s,\sigma) \varphi_{q_2}(\sigma) \sin(\sigma),
\end{split}
\end{equation}
from~\eqref{phi_j_q_series} we have
\begin{equation}
\label{eq:before_fc1}
\begin{split}
\mathcal{I}_{q_1,q_2}[\varphi_{q_2}](s) \sim  \frac{b_{q_2}-a_{q_2}}{2}&  \sum _{\ell=0}^{n}  \int_0^{\pi}  \log|r^\mathrm{out}(s) - \cos(\sigma)| [\alpha^{1}_{\ell}(s) \cos(\ell \sigma) +  \beta^{1}_{\ell}(s) \sin(\ell \sigma)] d\sigma  \\
+ \frac{b_{q_2}-a_{q_2}}{2}&  \sum _{\ell=0}^{n} \int_0^{\pi}  [\alpha^{2}_{\ell}(s) \cos(\ell \sigma) +  \beta^{2}_{\ell}(s) \sin(\ell \sigma)] d\sigma.
\end{split}
\end{equation}
A quadrature rule for~\eqref{eq:before_fc} now results from explicit evaluation of integrals of sines and cosines as well as integrals of the form~\eqref{log_ints} with $r = r^\mathrm{out}(s)$ (equations~\eqref{symms_operator} and~\eqref{eq:Am_formula} below).
\end{enumerate}
The integrals~\eqref{log_ints} can be produced in closed form for all
real values of $r$ (cf. Remark~\ref{remark:log_integrals}). The well
known expressions for the log-cosine integrals (Symms
operator)~\cite{masonchebyshev}
\begin{equation}
\label{symms_operator}
\begin{split}
\int_0^{\pi} \log|r  - \cos(\sigma)| \cos(n\sigma) d\sigma &= \frac{1}{2 n} \cos(n \arccos(r)) \quad\mbox{for}\quad n \neq 0,\\
\int_0^{\pi} \log|r  - \cos(\sigma)| d\sigma & = \frac{\log(2)}{2} \quad\mbox{for}\quad n=0
\end{split}
\end{equation}
are valid provided $ \left| r \right| \leq 1$. The recently derived
expression~\cite{AkhBrunoReitich}
\begin{equation}
\label{eq:Am_formula}
\begin{split}
\int_0^{\pi}\log (r-\cos (\sigma ))e^{i n \sigma }d\sigma=(-i)&\left[-\frac{1-\omega_1 ^n}{n}\log \left|1-\omega _1\right|+\frac{(-1)^n-\omega _1 ^n}{n}\log \left|1+\omega _1\right|\right. \\
+\frac{1}{n}\sum _{j=0}^{n-1} \left(\omega _1 ^j+\omega_2 ^j\right)&\frac{\left(1-(-1)^{n-j}\right)}{n-j}-\frac{1-\omega _2 ^n}{n}\log |1-\omega _2|\\
+\frac{(-1)^n-\omega _2 ^n}{n}\log |1+\omega _2|-&i\pi  \frac{\omega _2 ^n}{n}-\left.\frac{1}{n^2}\left[1-(-1)^n\right]+\log (2) \frac{1-(-1)^n}{n} \right] ,
\end{split}
\end{equation}
where $\omega_1$ and $\omega_2$ are the roots of the polynomial
\begin{equation}
2 \omega r - \omega ^2-1=-\left(\omega -\omega _1\right)\left(\omega -\omega _2\right),
\end{equation}
holds for all real values of $r$; the real and imaginary parts of this expression provide the necessary log-cosine and log-sine integrals.

In view of the high-order convergence of the FC method (cf. Section~\ref{sec:numerical_results} and Appendix~\ref{sec:FC}), a high-order accurate algorithm for evaluation of $\mathcal{I}_{q_1,q_2}[\varphi]$  (and thus $\mathcal{\widetilde{I}}_{q_1,q_2}[\widetilde{\varphi}]$) on the sole basis of a uniform $\sigma$ mesh results through application of equations~\eqref{symms_operator} and~\eqref{eq:Am_formula} in conjunction with equations~\eqref{eq:before_fc_repeated},~\eqref{eq:before_fc2} and ~\eqref{eq:before_fc1} . 

\begin{remark}\label{rem_FC_graded}
  In the following section we propose an algorithm that is applicable
  in the case $\Gamma$ is a non-smooth but piecewise smooth curve
  $\Gamma$. While the methods of that section can also be used for
  smooth curves $\Gamma$, the FC-based methods introduced in the
  present section are generally significantly more efficient for a
  given prescribed error and more accurate for a given discretization
  size. The improvements that result from use of the FC-based approach
  are demonstrated in section~\ref{sec:corner_geometries} by means of
  a variety of numerical results.
\end{remark}

\subsection{Graded-mesh algorithm for evaluation of the integral
  operators~\eqref{eq:parametrized_operators} \label{sec:corner_geometries}}

As can be seen by consideration of
equations~\eqref{asymptoticsDensitycorner}
and~\eqref{asymptoticsDensity}, the presence of corners in the domain
boundary affects significantly the singular character of the Zaremba
integral density. In order to accurately approximate our integral
operators for domains with corners we utilize a quadrature
method~\cite{sag1964numerical,KUSSMAUL:1969,MARTENSEN:1963,KRESS:1990,COLTON:1998}
which, based on changes of variables that induce graded meshes and
vanishingly small Jacobians, regularize the associated integrands at
corners and thus enable high order integration even in presence of 
density singularities.

\subsubsection{Graded-mesh algorithm: Polynomial change of variables \label{subsection_cov}} 
A set of quadrature weights similar to those given
in~\cite[p. 75]{COLTON:1998} are incorporated in the present context
to account accurately for the logarithmic singularity of the kernel
and the singularities of the integral density at corners.  As
in~\cite{COLTON:1998} a graded mesh on each of the intervals $[a_q,
b_q]$, $q=1,\dots,Q_D+Q_N$ is induced by means of a polynomial change
of variables of the form $\tau = w_q(\sigma)$, where
\begin{equation}
\label{eq:change_variable}
\begin{split}
w_q(\sigma)&=a_q + (b_q-a_q) \frac{[v(\sigma)]^p}{[v(\sigma)]^p+[v(2\pi-\sigma)]^p},\quad 0\leq \sigma \leq 2\pi,\\
v(\sigma)& = \left(\frac{1}{p}-\frac{1}{2}\right)\left(\frac{\pi-\sigma}{\pi}\right)^3+\frac{1}{p}\frac{\sigma-\pi}{\pi}+\frac{1}{2},
\end{split}
\end{equation}
and where $p\geq 2$ is an integer. Each function $w_q$ is smooth and
increasing in the interval $[0, 2\pi]$, and their $k$-th derivatives
satisfy $w_q^{(k)}(0) = w_q^{(k)}(2\pi)= 0$ for $1 \leq k \leq p -
1$. 

\begin{remark}\label{remark:cancellations_corner0}
In addition to change of variables~\eqref{eq:change_variable} and
associated graded meshes, the method~\cite{COLTON:1998} for
domain with corners (which is only applicable to the Dirichlet
problem) relies on a certain subtraction of values of the integral density at corner points times
a Gauss integral to provide additional regularization of the
integration process. The algorithms in this paper, which can be used
to treat all three, the Dirichlet, Neumann and Zaremba boundary value
problems, do not incorporate any such subtraction, however;
see~Remark~\ref{remark:cancellations_corner} for a brief discussion in
these regards.)
\end{remark}

In detail, the integrand in
equation~\eqref{eq:parametrized_operators} contains
singularities of various types, namely
\begin{enumerate}
\item\label{one} Singularities that result solely from corresponding
  singularities in the density $\widetilde{\varphi}$---in the term
  $\widetilde{K^1}(t,\tau) \log R^2(t,\tau)\widetilde{\varphi}(\tau)$
  for $q_2\ne q_1$ and in the term
  $\widetilde{K^2}(t,\tau)\widetilde{\varphi}(\tau)$ for both $q_2\ne
  q_1$ and $q_2= q_1$; and
\item\label{two} Combined singularities induced by the density and
  the logarithmic factor---in the term $\widetilde{K^1}(t,\tau) \log
  R^2(t,\tau)\widetilde{\varphi}(\tau)$ for $q_1 = q_2$.
\end{enumerate}
\begin{remark}
\label{near_singularity_no_problem}
Concerning point~\ref{one} above note that, although for $q_1\not=q_2$
the factor $\log R^2( t,\tau)$ is smooth, this term does give rise to
a logarithmic near-singularity for $t$ close to either $a_{q_2}$ or
$b_{q_2}$. It is easy to check, however, that the approach provided
below for treatment of the singular character of $\widetilde{\varphi}$
suffices to account with high-order accuracy for the near-logarithmic
singularity as well.
\end{remark}
Using the change of variables~\eqref{eq:change_variable} for both integration and observation variables, that is, setting $t = w_{q_1}(s)$ and $\tau = w_{q_2}(\sigma)$,  the integral~\eqref{eq:parametrized_operators} can be re-expressed in the form
\begin{equation}
\begin{split}
\label{eq:parametrized_operators_corner}
\mathcal{I}_{q_1,q_2}[\varphi](s) =   \int_{0}^{2\pi}&\widetilde{K^1}(w_{q_1}(s),w_{q_2}(\sigma)) \log R^2( w_{q_1}(s),w_{q_2}(\sigma))\varphi_{q_2}(\sigma)  w_{q_2}'(\sigma) \de \sigma +\\
\int_{0}^{2\pi}&\widetilde{K^2}(w_{q_1}(s),w_{q_2}(\sigma)) \varphi_{q_2}(\sigma)  w_{q_2}'(\sigma) \de \sigma,
\end{split}
\end{equation}
where  $\varphi_{q_2}(\sigma) = \widetilde{\varphi}(w_{q_2}(\sigma))$. This procedure effectively treats the density singularities mentioned in point~\ref{one} above. Indeed, since values $p \ge 2$ are used for the parameter $p$ in equation~\eqref{eq:change_variable} and given the singular character~\eqref{asymptoticsDensitycorner} of the density $\widetilde{\varphi}$, the product $\varphi(\sigma)  w_{q_2}'(\sigma)$  is smoother than $\widetilde{\varphi}$: this product can be made to achieve any finite order of differentiability by selecting $p$ large enough. 

To deal with the singularities mentioned in point 2 above, on the other hand, we utilize the following notations: for $q_1 = q_2 = q$, we let
\begin{equation}\label{transformatin_kernels_corner}
\begin{split}
K^1_{q,q}(s,\sigma) &= \widetilde{K^1}(w_{q}(s),w_{q}(\sigma)), \\
K^2_{q,q}(s,\sigma)&=\widetilde{K^1}(w_{q}(s),w_{q}(\sigma)) \log\left(\frac{R^2(w_{q}(s),w_{q}(\sigma))}{4\sin^2\frac{s-\sigma}{2}}\right) + \widetilde{K^2}(w_{q}(s),w_{q}(\sigma)). 
\end{split}
\end{equation}
Note that the  ``diagonal term''  that occurs in the kernel $K^2_{q,q}$ for $s=\sigma$ is given by $K^2_{q,q}(s,s)=2 \widetilde{K^1}(w_{q}(s),w_{q}(s)) \log(w_{q}'(s)|z'(w_{q}(s))|)+\widetilde{K^2}(w_{q}(s),w_{q}(s))$. Using these transformations the integrals~\eqref{eq:parametrized_operators_corner} for $q_1 = q_2 = q$ can be re-expressed in the form
\begin{equation}
\label{eq:parametrized_operators_corner_reexpressed}
\mathcal{I}_{q,q}[\varphi_{q}](s)  =   \int_{0}^{2\pi} K^1_{q,q}(s,\sigma)  \log(4\sin^2\frac{s-\sigma}{2}) \varphi_{q}(\sigma)  w_{q}'(\sigma) \de \sigma+ \int_{0}^{2\pi} K^2_{q,q}(s,\sigma) \varphi_{q}(\sigma)  w_{q}'(\sigma) \de \sigma.
\end{equation}

\subsubsection{Graded-mesh algorithm: Discretization and quadratures \label{subsection_discretization_corner}}

In view of the discussion presented in Section~\ref{subsection_cov}
our overall numerical algorithm for evaluation of the
integrals~\eqref{eq:parametrized_operators_corner} (and thus
~\eqref{eq:parametrized_operators}) proceeds through separate
consideration of the cases $q_1=q_2=q$ and $q_1\not=q_2$. In the case
$q_1 = q_2 = q$ we utilize the
expression~\eqref{eq:parametrized_operators_corner_reexpressed}: the
first (resp. second) integral in this equation is evaluated by means
of the logarithmic quadrature~\eqref{eq:logarithmic_quadrature_corner}
below (resp. the spectrally accurate trapezoidal
rule~\eqref{eq:trapezoidal_rule} below). For the case $q_1\not=q_2$,
on the other hand, we use the
expression~\eqref{eq:parametrized_operators_corner} directly: we
combine both integrals into one which is then evaluated by means of
the trapezoidal rule~\eqref{eq:trapezoidal_rule}. The logarithmic and
trapezoidal rules mentioned above proceed as follows:
\begin{itemize}
\item{Logarithmic quadrature ($q_1=q_2 = q$)}.

We consider integrals whose integrand, like the one in the first integral in equation~\eqref{eq:parametrized_operators_corner_reexpressed}, consists of a product of a smooth $2\pi$-periodic function $f$ times the logarithmic factor $\log\left(4\sin^2\frac{s-\sigma}{2}\right)$. Such integrals are produced with spectral accuracy by means of the rule
\begin{equation}
\label{eq:logarithmic_quadrature_corner}
\int_0^{2\pi}f(\sigma)\log\left(4\sin^2\frac{s-\sigma}{2}\right)\de \sigma\sim \sum_{j=1}^{2n}R_j^{(n)}(s)f(\sigma_j),
\end{equation}
where $\sigma_j=(j-1)\pi/n$, $n\in\mathbb{N}$ and where the quadrature weights $R_j(s)$ are given by~\cite[p. 70]{colton1984}
\begin{equation*}
R_j(s)  = -\frac{2\pi}{ n }\sum_{m=1}^{n-1}\frac{1}{m}\cos m(s-\sigma_j)-\frac{\pi}{n^2}\cos n(s-\sigma_j).
\end{equation*}
Following~\cite{bruno2012second} we note that, letting
\begin{equation*}
R_k = -\frac{2\pi}{n}\sum_{m=1}^{n-1}\frac{1}{m}\cos\frac{m k \pi}{n}-\frac{(-1)^k\pi}{n^2}
\end{equation*}
we have $R_j(\sigma_i) = R_{|i-j|}$---so that the weights $R_j(\sigma_i)$ can be evaluated rapidly by means of Fast Fourier Transforms.

\item{Trapezoidal rule.} 

As is well known, spectrally accurate integrals of smooth $2\pi$-periodic functions $f$ can be obtained by means of the trapezoidal rule 
\begin{equation}
\int_0^{2\pi}f(\sigma)\de\sigma \sim\frac{\pi}{n}\sum_{j=1}^{2n}f(\sigma_j),\label{eq:trapezoidal_rule}
\end{equation}
where again $\sigma_j=(j-1)\pi/n$. 
\end{itemize}

\begin{remark}
\label{remark:cancellations_corner}
With reference to Remark~\ref{remark:cancellations_corner0},
subtraction of a certain multiple of a Gauss integral can be used in
the case of the Dirichlet problem to somewhat mollify corner
singularities and thereby enhance the convergence of the numerical
integration method. Considering the
expressions~\eqref{asymptoticsDensitycorner} for the singularities in
the density functions, even without the subtraction the method
described above in this section is easily checked to be consistent
with the system~\eqref{modelC} of integral equations for sufficiently
large value of $p$. Although a proof of the stability of the method is
left for future work, the numerical results in this paper (see e.g.
Figure~\ref{fig:fc_vs_ck}) strongly suggest that perfect stability
results from this approach. As the value of $\alpha$ grows, however,
the minimum required value of $p$ grows as well, thereby increasing the
condition number of the system. This difficulty can alternatively be
addressed by means of a singularity resolution methodology introduced
in~\cite{bruno2009high}---a full development of which is beyond the
scope of this paper and which is thus left for future work.
\end{remark}

\subsection{Discrete boundary integral operator \label{Section:discretization_matrix}}

This section presents the main algorithm for evaluation of the discrete version of the form~\eqref{modelD} of the boundary operator~\eqref{modelC1} (cf.~\eqref{modelB} or, equivalently,~\eqref{eq:integrals_decomposed}); the discretization procedure relies on use of  the high-order quadrature methods described in Sections~\ref{sec:smooth_geometries} and~\ref{sec:corner_geometries} for  evaluation of each on of the integrals in equation~\eqref{eq:integrals_decomposed}.

We denote by $n_q$ the number of discretization points used on the boundary segment $\Gamma_q$, $q=1,\dots,Q_D+Q_N$ and we call 
\begin{equation}\label{N_sum_nq}
n = \sum_{q=1}^{Q_D+Q_N} n_q
\end{equation}
the number of discretization points used throughout $\Gamma$. The discrete algorithms introduced in this paper rely on use of the uniform grids 
\begin{equation}\label{sigma_q_j}
\sigma_{q}^j =(j-1)\gamma \pi/n_q\quad ,\quad j=1,\dots, n_q
\end{equation}
in the interval $0\leq \sigma \leq \gamma\pi$, where $\gamma=1$ for the FC-based algorithm (Section~\ref{sec:smooth_geometries}) and $\gamma = 2$ for the graded-mesh algorithm (Section~\ref{sec:corner_geometries}). The corresponding points $\tau_{q}^j$
 in the parameter space are given by $\tau_{q}^j = \xi_{q}(\cos(\sigma_{q}^j))$ in the FC-based algorithm (see equation~\eqref{linear_transformation_smooth_case}), and by $\tau_{q}^j = w_{q}(\sigma_{q}^j)$ in graded-mesh algorithm. 

\begin{remark}\label{remark:select_nq}
  The following procedure is suggested for determination of the values
  of the parameters $n_q$ mentioned above. Given a desired meshsize
  $h\in\mathbb{R}$ (which should be selected so as to appropriately
  discretize the highest spatial oscillations under consideration) we
  take $n_q=\max \{n_h,n_0\}$ where $n_h$ is the smallest integer for
  which the distance between any two consecutive points in $\Gamma_q$
  is not larger than $h$, and where $n_0$ is an integer whose role is
  to ensure that the number of discretization points in each boundary
  segment is not less than the minimum number of discretization points
  required by the method used (either the Fourier Continuation method,
  see Appendix~\ref{sec:FC}, or the graded-mesh algorithm,
  cf. equation~\eqref{eq:change_variable}) to guarantee the desired
  convergence rate takes place.
\end{remark} 
\begin{remark}\label{phi_only_discrete_values}
  We point out that in both the FC-based and graded-mesh algorithms
  (Sections~\ref{sec:smooth_geometries}
  and~\ref{sec:corner_geometries}, respectively) the approximations of
  the values $\mathcal{I}_{q_1,q_2}[\varphi_{q_2}](\sigma_{q_1}^j)$
  used by our algorithms only depend on values of $\varphi_{q_2}$ at
  the points $\sigma_{q_2}^{j}$, $j=1,\dots,n_{q_2}$. In the smooth
  domain case this indeed results from the fact that an $m$-th order
  Fourier continuation $f^c$ of a function $y=f(x)$ only depends on
  the values of $f$ at the discretization points $x_i$ (see
  Appendix~\ref{sec:FC} and take into account
  equations~\eqref{phi_j_q_series},~\eqref{eq:before_fc_repeated},~\eqref{eq:before_fc2}
  and~\eqref{eq:before_fc1}). For the graded-mesh case, in turn, this
  follows from the fact that the quadrature
  rules~\eqref{eq:logarithmic_quadrature_corner},~\eqref{eq:trapezoidal_rule}
  only use values of the density $\varphi_{q_2}$ at the points
  $\sigma^j_{q_2}$.
\end{remark}

In view of Remark~\ref{phi_only_discrete_values} and
equation~\eqref{N_sum_nq} and associated text, a discrete version of
the integral density $\psi$ can be obtained in the form of an
$n$-dimensional vector of unknowns
\begin{equation*}
{\tt c}  = \begin{bmatrix} {\tt c}_{1} \\ \dots \\ {\tt c}_{Q_N+Q_D} \end{bmatrix}
\end{equation*}
where ${\tt c}_{q}$ is a sub-vector of length $n_q$ which contains the
approximate unknown density values at the points $\sigma_{q}^j$: ${\tt
  c}_{q}^j \sim \varphi_{q}(\sigma_{q}^j)$.  An approximate boundary
operator~\eqref{modelD} based on either the FC method (for smooth
curves $\Gamma$) or the graded mesh method (for smooth or non-smooth
curves $\Gamma$) can thus be obtained in the form of a matrix
\begin{equation}\label{A_MU}
{\tt A}_{\mu} = \begin{bmatrix} ({\tt A}_{\mu})_{1,1} & ({\tt A}_{\mu})_{1,2} & \dots & ({\tt A}_{\mu})_{1,Q_d+Q_n} \\ \dots \\ ({\tt A}_{\mu})_{Q_D+Q_N,1} & ({\tt A}_{\mu})_{Q_D+Q_N,2} & \dots & ({\tt A}_{\mu})_{Q_D+Q_N,Q_D+Q_N} \end{bmatrix}.
\end{equation} 
Here the sub-blocks $({\tt A}_{\mu})_{q_1,q_2}$ are discrete operators
which for $q_1 = q_2 =q\in J_N$ approximate the continuous operators
$-I/2 + \mathcal{I}_{q,q}$ (where $I$ is the identity operator):
\begin{equation}\label{eq:neumann_boundary_jump}
-\frac{\varphi_{q}}{2} + \mathcal{I}_{q,q}[\varphi_{q}](\sigma_{q}^j ) \sim \sum_{j=1,n_{q}} ({\tt A}_{\mu})_{q,q}^{i,j} {\tt c}_{q}^j  , \qquad i=1,\dots,n_{q},
\end{equation} 
and which for all other pairs of indexes $q_1,q_2 = 1,\dots,Q_D+Q_N$
approximate the continuous
operators $\mathcal{I}_{q_1,q_2}$:
\begin{equation*}
 \mathcal{I}_{q_1,q_2}[\varphi_{q_2}](\sigma_{q_2}^j ) \sim\sum_{j=1,n_{q_2}} ({\tt A}_{\mu})_{q_1,q_2}^{i,j} {\tt c}_{q_2}^j  , \qquad i=1,\dots,n_{q_1}.
\end{equation*} 
In cases in which an overall FC-based method is used the blocks
$({\tt A}_{\mu})_{q,q}$ are matrices which encapsulate the various
integration methods described in Section~\ref{sec:smooth_geometries};
if the graded-mesh method is used instead then the blocks $({\tt
  A}_{\mu})_{q,q}$ collect the contributions produced by the
quadrature methods presented in Section~\ref{sec:corner_geometries}.

Details of the algorithm used to produce the blocks $({\tt
  A}_{\mu})_{q_1,q_2}$ are given in
Algorithms~\ref{matrix_construction_smooth}
and~\ref{matrix_construction_corner} below. The input parameters in
these algorithms are to be selected in accordance with
Remarks~\ref{remark:cases} and~\ref{remark:select_nq}.

\begin{algorithm}[h!]
\caption{Construction of the matrix block $({\tt A}_{\mu})_{q_1,q_2}$ for the FC-based algorithm}\label{matrix_construction_smooth}
\begin{algorithmic}[1]
\State 
Input $q_1$, $q_2$, $n_{q_1}$ and $n_{q_2}$ (Section~\ref{Section:discretization_matrix}).
\For{$j_2=1: n_{q_2}$}
\State Let ${\tt c}_{q_2}^{j_2}=1$ and  ${\tt c}_{q_2}^{j}=0$ for $j=1, n_{q_2}, j\not = j_2$ (cf. Remark~\ref{phi_only_discrete_values}).
\For{$j_1=1: n_{q_1}$} 

\State 
With reference to Cases 1 through 3 in Section~\ref{ker_dec},
calculate $\phi_{q_1,q_2}^1(\sigma_{q_1}^{j_1},\sigma_{q_2}^{j_2})$ and $\phi_{q_1,q_2}^2(\sigma_{q_1}^{j_1},\sigma_{q_2}^{j_2})$ using eq.~\eqref{phi_j_q_1}  in case 1, eq.~\eqref{phi_j_q_2} in case 2 and eq.~\eqref{phi_j_q_3} in case 3. 

\State
Calculate the coefficients $\alpha^1_{\ell},\beta^1_{\ell},\alpha^2_{\ell},\beta^2_{\ell}$ of the Fourier Continuation expansions~\eqref{phi_j_q_series} using the FC algorithm (see Appendix~\ref{sec:FC}).

\If {$q_1=q_2 =:q$}
\State  Evaluate $({\tt A}_{\mu})_{q,q}^{j_1, j_2}$ using the approximation in eq.~\eqref{eq:before_fc2}.

\If {$q \in J_N \mbox{ and } j_1=j_2 $}
\State Add the identity part corresponding to the jump of the density (eq.~\eqref{eq:neumann_boundary_jump}).
\EndIf

\Else

\If {$\tau_{q_1}^{j_1}$ is ``far'' from the interval $[a_{q_2},b_{q_2}]$ (condition~\eqref{far_close_condition}) }

\State  Evaluate $({\tt A}_{\mu})_{q_1,q_2}^{j_1, j_2}$ 
using the approximation in eq.~\eqref{eq:before_fc_repeated}.

\Else

\State  Evaluate $({\tt A}_{\mu})_{q_1,q_2}^{j_1, j_2}$ 
using the approximation in eq.~\eqref{eq:before_fc1}.

\EndIf

\EndIf

\EndFor
\EndFor

 \end{algorithmic} \end{algorithm}

\begin{algorithm}[h!]
\caption{Construction of the matrix block $({\tt A}_{\mu})_{q_1,q_2}$ for graded-mesh algorithm}\label{matrix_construction_corner}
\begin{algorithmic}[1]
\State
Input $q_1$, $q_2$, $n_{q_1}$ and $n_{q_2}$ (Section~\ref{Section:discretization_matrix}).
\For{$j_2=1:n_{q_2}$}
\State 
Let ${\tt c}_{q_2}^{j_2}=1$ and  ${\tt c}_{q_2}^{j}=0$ for $j=1, n_{q_2}, j\not = j_2$  (cf. Remark~\ref{phi_only_discrete_values}).
\For{$j_1=1: n_{q_1}$}
\If {$q_1=q_2 =:q$}
\State  Evaluate $({\tt A}_{\mu})_{q,q}^{j_1, j_2}$ 
using the decomposition~\eqref{eq:parametrized_operators_corner_reexpressed} via combination of the logarithmic quadrature~\eqref{eq:logarithmic_quadrature_corner} and the trapezoidal rule~\eqref{eq:trapezoidal_rule}.

\If {$q \in J_N \mbox{ and } j_1=j_2$}
\State Add the identity part corresponding to the jump of the density (eq.~\eqref{eq:neumann_boundary_jump}).
\EndIf
\Else
\State  Evaluate $({\tt A}_{\mu})_{q_1,q_2}^{j_1, j_2}$ 
using the decomposition~\eqref{eq:parametrized_operators_corner} and the trapezoidal rule~\eqref{eq:trapezoidal_rule}.
\EndIf
\EndFor
\EndFor

 \end{algorithmic} \end{algorithm}
\section{Eigenvalue search \label{Section:EVP}}

This section presents an efficient algorithm for eigenvalue
search---which, in the context of the present paper, amounts to search
for values of $\mu$ in a given range $[\mu_{min}, \mu_{max}]$ for
which the statement~\eqref{modelC} is satisfied. The search algorithm
presented below can be utilized in conjunction with any numerical
discretization of the operator~\eqref{modelC1} and, indeed, it can be
applied to integral formulations of more general eigenvalue
problems. Naturally, however, in this paper we apply our search
algorithm in combination with the discrete version ${\tt A}_{\mu}$ of
the operator $\mathcal{A}_\mu$ (cf. equations~\eqref{modelC1}
and~\eqref{A_MU}) which results from suitable applications of the
quadrature rules presented in
Sections~\ref{bound_dec}--\ref{Section:discretization_matrix} to the
operators~\eqref{eq:parametrized_operators}.
 
\begin{figure}[h!]
\centering
  \includegraphics[width=5in]{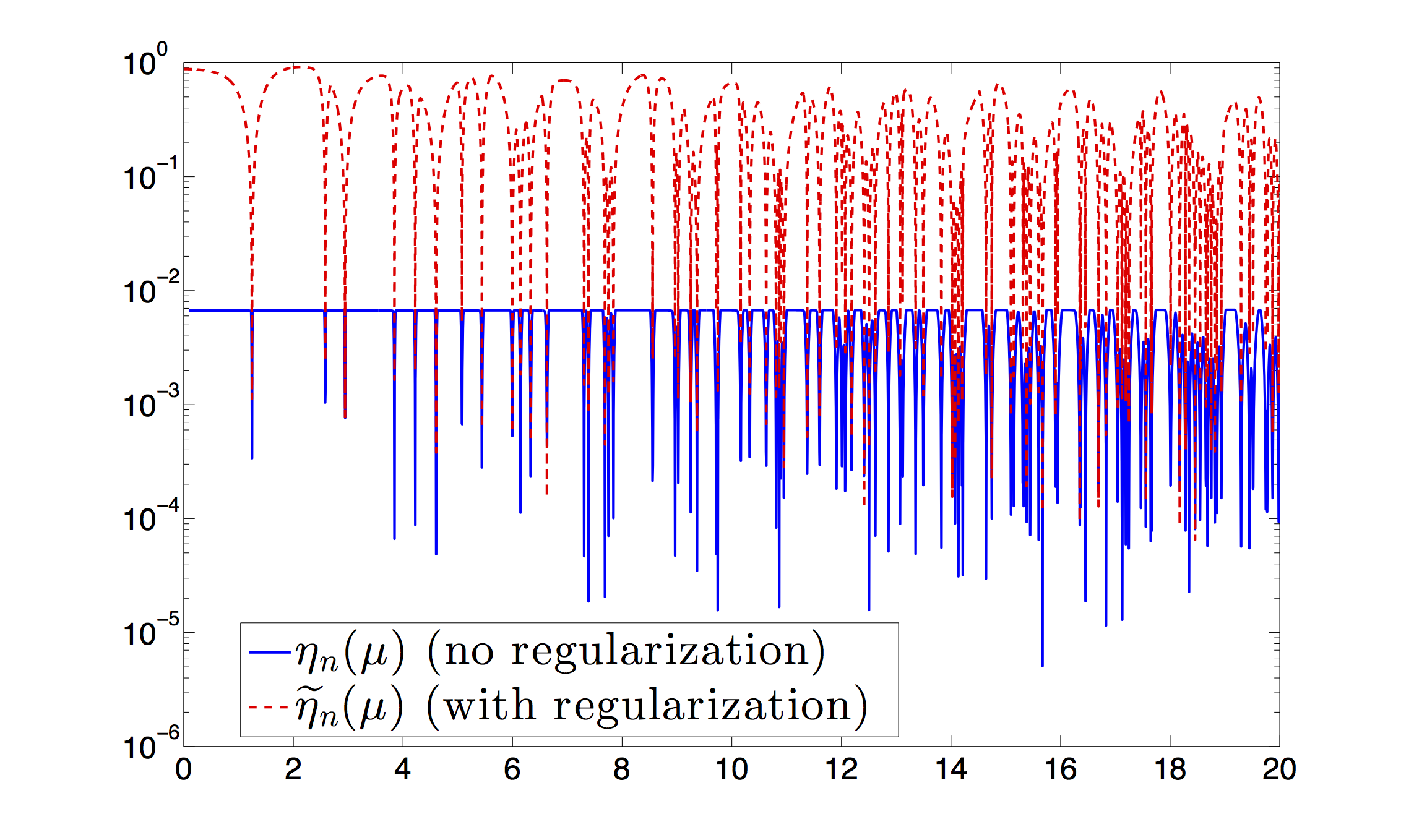}
 \caption{Comparison between $\eta_n(\mu)$ and  $\widetilde{\eta}_n(\mu)$}
  \label{fig:comparison}
\end{figure}
	
\begin{figure}[h!]
\centering
\includegraphics[width=3in]{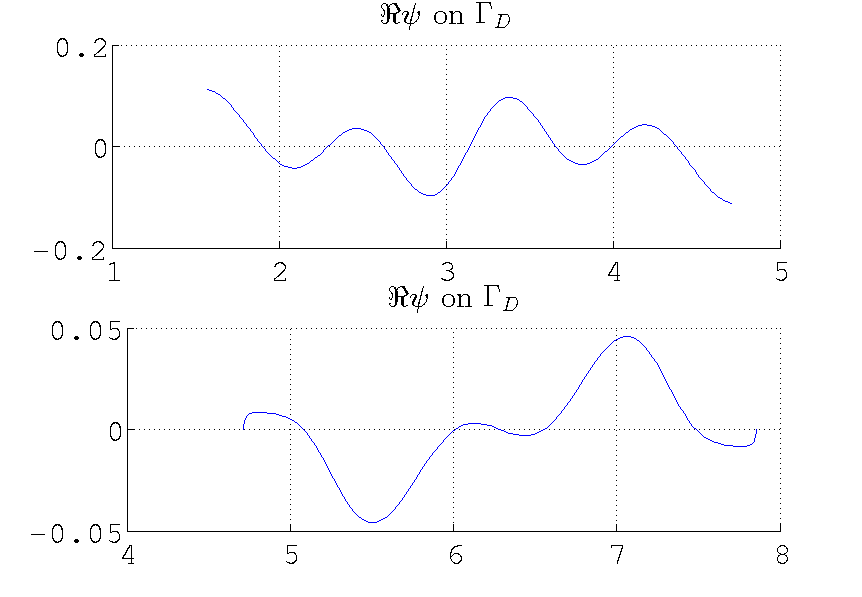}
\includegraphics[width=3in]{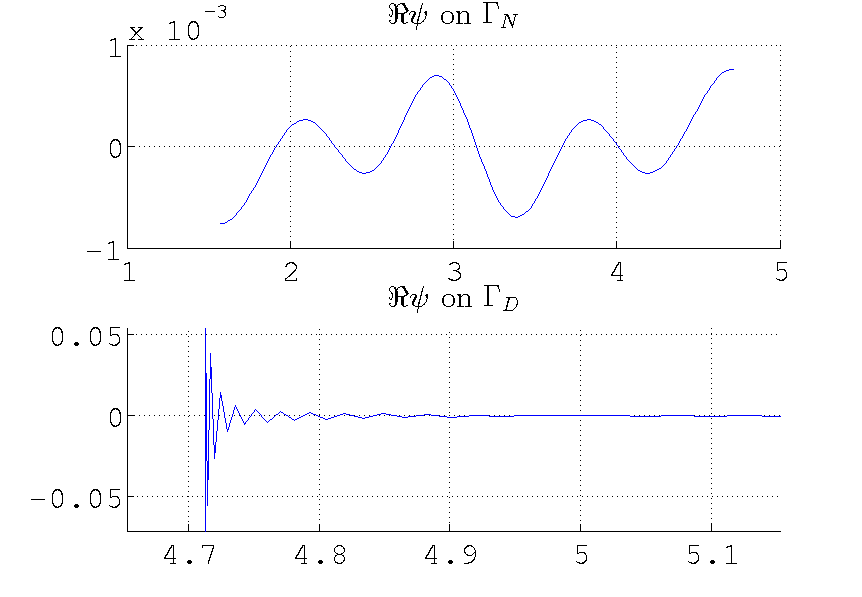}
\caption{Densities (singular vectors) corresponding to smallest
  singular values for a formulation without regularizing interior
  points. Left column, $\mu^2$ is an eigenvalue: vanishingly small
  values of the singular value $\eta_n(\mu)$ result for densities that
  are not rapidly oscillatory. Right column: $\mu^2$ is not an
  eigenvalue. Note the oscillations on the Dirichlet portion of the
  density (lower-right image) which give rise to a small singular
  value $\eta_n(\mu)$ even in this case in which $\mu^2$ is not a
  Dirichlet-Neumann eigenvalue.}
\label{fig:eigenvector_smooth}
\end{figure}
\subsection{Discussion \label{Section:simple scan}}

In view of~\eqref{modelC} and associated text, the eigenvalues $\lambda$ in equation~\eqref{modelA} can be approximated by the squares of the values $\mu$ for which the corresponding matrix ${\tt A}_{\mu}$ is not invertible. Thus all approximate eigenvalues $\lambda_j=\mu_j^2$ of the problem~\eqref{modelA} in a given interval $[\lambda_{min}, \lambda_{max}]$ can be obtained from the values of $\mu\in [\sqrt{\lambda_{min}}, \sqrt{\lambda_{max}}]$  for which  the minimum singular value $\eta_n(\mu)$ of the matrix ${\tt A}_{\mu}$ equals zero---or is otherwise sufficiently  close to zero.  

Unfortunately, this approach presents significant challenges in
practice---as was noted in~\cite{duran_nedelec,steinbachDirichlet} in
connection with applications to related Dirichlet problems for the
Laplace equation (but cf. Remark~\ref{first_kind}). The difficulty is
demonstrated in Figure~\ref{fig:comparison} (solid curve) which
displays the function $\eta_n(\mu)$ for values of $\mu$ in the
interval $[0, 20]$ for the Zaremba eigenproblem~\eqref{modelA} on a
unit disc (where Dirichlet and Neumann boundary conditions are
prescribed on the upper and lower halves of the disc
boundary). Clearly, the function $\eta_n(\mu)$ stays at a nearly
constant level except for narrow regions around minima. This makes the
derivative of $\eta_n(\mu)$ nearly 0 throughout most of the search
domain, and, thus, renders efficient application of root-finding
methods virtually impossible.

The occurrence of this adverse characteristic of the function
$\eta_n(\mu)$ can be explained easily by consideration
of~\eqref{modelC} and associated text. Indeed, in view of the
Riemann-Lebesgue lemma, arbitrarily small values of $\mathcal A_{\mu}
\psi$ can be obtained by selecting densities $\psi$ leading to
functions $\phi_{q_1,q_2}^j$ (equation~\eqref{phi_j_q_series}) which
equal highly oscillatory functions of $\sigma$ on the Dirichlet
boundary portion $\Gamma_D$ and which are close to zero on the Neumann
boundary portion $\Gamma_N$; see
Figure~\ref{fig:eigenvector_smooth}. At the discrete level, further,
for any given mesh-size $n$ only oscillatory functions up to a certain
maximal oscillation level are supported. Consequently, as $n$ (and
therefore the maximal oscillation level) are increased, the minimum
singular value $\eta_n(\mu)$ (which equals the minimum mean-square
norm of $\tt A_{\mu} {\tt c}$ for ${\tt c}$ in the unit sphere) itself
decays like $1/n$, without significant dependence on $\mu$---except
for cases that correspond to actual eigenvalues. In order to devise a
solution for this problem we note that the continuous analog of our
minimum singular value (namely, the infimum of $||\mathcal A_{\mu}
\psi||$ over all densities $\psi$ of unit norm) is actually equal to
zero for all values of $\mu$. But, naturally, a minimizing sequence
$\psi_k$ for which the operator values approach this infimum gives
rise to single-layer potentials $u_k$ that approach zero within
$\Omega$ as well---and, thus, such sequences $u_k$ do not approach
true eigenfunctions. A solution strategy thus emerges: a normalization
for the values of the single layer potential $u$ (eq.~\eqref{ansatz})
{\em in the interior of $\Omega$} can be used to eliminate such
undesirable minimizing sequences.  Details on possible implementations
of this strategy are presented in the following section.

\begin{remark}\label{first_kind}
  From the discussion above in this section it is easy to see that
  difficulties associated with highly oscillatory integrands only
  occur in cases in which the boundary integral operator is entirely
  or partially of the first kind: for second-kind integral equations
  such complications do not
  arise~\cite{zhao2014robust,kuo2000applications}. We note, however,
  that use of (partial or full) first-kind formulations can be highly
  advantageous in some cases (such as, e.g., for the problems
  considered in this paper!) for which use of second-kind equations
  would necessarily require inclusion of hypersingular
  operators---which are generally significantly more challenging from
  a computational perspective; see e.g.~\cite{bruno2013high}. The
  normalization techniques mentioned in Section~\ref{Introduction} and
  discussed in more detail in Section~\ref{Section:modified
    algorithm} completely resolves the difficulty arising from use of
  first-kind formulations and enables successful use of
  numerically-well-behaved, easy-to-use first-kind equations for
  solution of eigenvalue problems for general domains.
\end{remark}

\subsection{Eigenfunction normalization. \label{Section:modified algorithm}}
The difficulties outlined in the previous section can be addressed by consideration of a modified discrete system of equations which, by enforcing an appropriate normalization in the domain interior, as it befits eigenfunctions of a differential operator, prevents oscillatory vectors ${\tt c}$ to give rise to small values of the  ${\tt A}_\mu {\tt c}$ unless $\mu$ corresponds to an actual eigenvalue; cf. Section~\ref{Section:simple scan}. 
\begin{figure}[h!]
\centering
  \includegraphics[width=2in]{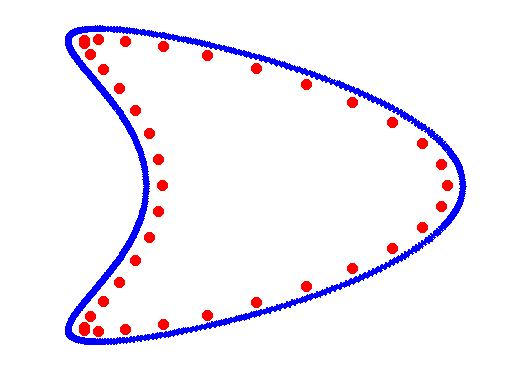}
  \caption{Interior curve $\Gamma_i$. An adequate discretization of
    $\Gamma_i$, possibly significantly coarser than the discretization
    on the boundary curve $\Gamma$, is used to penalize vanishingly
    small Laplace eigenfunctions.}
  \label{fig:dual_surface}
\end{figure}

To enforce such a normalization we consider equations~\eqref{ansatz}
and~\eqref{modelB}, and we define an additional operator ${\mathcal
  A}^{(3)}_\mu$ by
\begin{equation}\label{ansatz2}
\mathcal{A}^{(3)}_\mu[ \psi](x)  =  \displaystyle \int _{\Gamma}G_\mu(x,y)\psi (y)ds_y \qquad \mbox{for } x \in \Gamma_i,
\end{equation}
where $\Gamma_i \subseteq \Omega$ is an adequately selected set of points in the interior of $\Omega$. 
A natural choice is given by $\Gamma_i = \Omega$, in such a way that the normalization condition becomes $\displaystyle \int _{\Omega} |u|^2 dx = 1 $. Other normalizations can be used, however, which lend themselves more easily to discretization. For example, letting $\Gamma_i$ be a curve roughly parallel to $\Gamma$ at a distance no larger than $\lambda_u/2$, cf. Figure~\ref{fig:dual_surface} (where $\lambda_u$ denotes the eigenfunction ``wavelength''  $\lambda_u= 2\pi /\mu$), one might equivalently prescribe
\begin{equation}\label{constraint}
\displaystyle \int _{\Gamma_i} |u|^2 d\ell = 1.
\end{equation} 
Indeed, given that $\Gamma_i$ is at a distance no larger than $\lambda_u/2$ from $\Gamma$, we expect that 
\begin{equation}\label{gammai_non_zero}
\mbox{``The eigenfunction must be nonzero in a subset of $\Gamma_i$ of  positive measure''.}
\end{equation} 
In the case of Dirichlet boundary conditions this statement is strongly supported by the eigenvalue bounds put forth  in~\cite{borisov2010asymptotics} and by the discussion in~\cite{woodworth1991derivation}. In the case of the Zaremba Dirichlet-Neumann boundary conditions we have not as yet found a corresponding theoretical discussion, but, in view of strong numerical evidence, throughout this paper we nevertheless assume~\eqref{gammai_non_zero} holds.

\begin{remark}\label{continuous_appended_operator}
It is useful to note that, assuming~\eqref{gammai_non_zero}, there is no non-zero density $\psi$ for which the equations ${\mathcal  A}_\mu[\psi] = 0$ and ${\mathcal  A}^{(3)}_\mu[\psi] = 0$ hold simultaneously.  Indeed, if these null conditions hold,~\eqref{gammai_non_zero} implies that the function $u$ defined by~\eqref{ansatz} vanishes throughout $\Omega$. In view of the uniqueness of solution of the Helmholtz equation in a exterior domain, further, we conclude that $u$ vanishes throughout $\mathbb{R}^2$.  Taking into account the jump conditions for the normal derivative of the single layer potential this implies that $\psi=0$, as desired.
\end{remark}

A discrete version of the normalization condition~\eqref{constraint} can be obtained by means of a suitable, possibly equispaced discretization $\{ x_j, j=1,m\} \subseteq \Gamma_i$ together with an associated discrete operator ${\tt B}_\mu$ which, based on the quadrature rules for smooth integrands described in Section~\ref{Section:eigenfunctions}, approximates values of ${\mathcal  A}^{(3)}_\mu$ at the points $x_j$:
\begin{equation}\label{eq:b_mu}
{\tt B}_\mu{\tt c}  \sim [u(x_j)].
\end{equation}
Defining the rectangular matrix
\begin{equation}
{\tt C}_\mu = \begin{bmatrix} {\tt A}_\mu \\ {\tt B}_\mu \end{bmatrix},
\label{rectangular}\end{equation}
in the present discrete context (and for a sufficiently fine discretization $\{ x_j, j=1,m\} \subseteq \Gamma_i$) Remark~\eqref{continuous_appended_operator} tells us that  the columns of the matrix ${\tt C}_\mu$ ought to be linearly independent. The normalization condition can be enforced by utilizing a QR-factorization 
\begin{equation*}
{\tt C}_\mu = \neQ \neR;
\end{equation*}
in accordance with equation~\eqref{rectangular}, further, we express the matrix $\neQ$ in terms of matrices comprising of its first $n$ rows and the remaining $m$ rows:
\begin{equation}\label{eq:q_a}
 \neQ = \begin{bmatrix} \neQ_A \\ \neQ_B \end{bmatrix}.
\end{equation}
(In a  related but different context, a QR factorization was used in~\cite{trefethen2005} to reduce or even eliminate difficulties associated with the method of particular solution for evaluation of Laplace eigenvalues; see Remark~\ref{tref_betcke} for details.)

The linearly independent columns of the matrix ${\tt C}_\mu$ are in
fact (discrete) approximate solutions of the Helmholtz equation
evaluated at the boundary points and the points on $\Gamma_i$. And, so
are the columns of the matrix $\neQ$, since they equal linear
combinations of the columns of ${\tt C}_\mu$. Thus, letting ${\tt d}$
denote a singular vector of the matrix $ \neQ_A$ ($\parallel {\tt
  d}\parallel = 1$) corresponding to a singular value equal to zero, $
\neQ_A {\tt d} = 0$ (for which we must necessarily have $\parallel
\neQ_B {\tt d} \parallel = 1$), and letting ${\tt c} = \neR^{-1}{\tt
  d}$, the product ${\tt C}_\mu {\tt c} = \neQ {\tt d}$ equals a
linear combination of the columns of ${\tt C}_\mu$ which vanishes on
$\Gamma$ and for which, therefore, the mean square on $\Gamma_i$
equals one. From the discussion above in this section it follows that
${\tt c}$, which is a discrete version of the density $\psi$, yields,
via a discrete version of the representation~\eqref{ansatz}
(Section~\ref{Section:eigenfunctions}) an approximate eigenfunction of
the problem~\eqref{modelA}.

Thus, relying on the Singular Value Decomposition (SVD) of the matrix
$\neQ_A$ for a given value of $\mu$, and calling
$\widetilde{\eta}_n(\mu)$ the smallest of the corresponding singular
values, 
\begin{equation}
\label{eq:sigma}
\widetilde{\eta}_n(\mu) =  \min_{{\tt b}\in \R^{n}, \parallel{\tt b}\parallel=1} \parallel\neQ_A(\mu) {\tt b} \parallel =  \parallel\neQ_A(\mu) {\tt d} \parallel,
\end{equation}
the proposed eigensolver is based on finding values of $\mu$ for which
$\widetilde{\eta}_n(\mu)$ is equal to zero.  A pseudocode for this
method is presented in Algorithm~\ref{halmos}.

\begin{algorithm}[H]
  \caption{Numerical evaluation of all $\mu\in [F_{min},F_{max}]$ for
    which~\eqref{modelC} is satisfied. }\label{halmos}

 \begin{algorithmic}[1]
 
   \State Read input parameters $h$
   (cf. Remark~\ref{remark:select_nq}) and $N_0$ (number of singular
   values in the wave number search range $ [F_{min},F_{max}]$
   actually to be produced via SVD).

   \For{$j=1 : N_0$} \State Set $\mu :=F_{min}+ j
   \frac{F_{max}-F_{min}}{N_0}$.  \State Construct the matrix of the
   discrete operator ${\tt A}_{\mu}$
   (Algorithms~\ref{matrix_construction_smooth}
   and~\ref{matrix_construction_corner}).  \State Construct the matrix
   of the discrete operator ${\tt B}_{\mu}$ (eq.~\eqref{eq:b_mu}).
   \State Compute the QR-factorization of the augmented system ${\tt
     C}_\mu$ (eq.~\eqref{rectangular}).  \State Compute the minimal
   singular value $\sigma_n(\mu_j)$ of $\neQ_A$
   (cf. eq.~\eqref{eq:q_a}).
\EndFor     
\State Utilizing the computed values of $\sigma_n(\mu_j)$ execute the root-finding algorithm mentioned in Section~\ref{flipping} to produce approximate roots of the function $\sigma(\mu)$.
 \end{algorithmic} \end{algorithm}

\begin{remark}\label{tref_betcke}
  As mentioned in the introduction, the method of particular solutions
  (MPS) relies on use of Fourier-Bessel series that match homogeneous
  Dirichlet boundary conditions to produce Laplace eigenvalues and
  eigenfunctions. A modified version of the MPS, which was introduced
  in reference~\cite{trefethen2005}, alleviates some difficulties that
  occur in the original version of the method by enforcing that, as is
  necessary in our case as well, the proposed eigenfunctions do not
  vanish (and, indeed, are normalized to unity) in some finite set of
  points in the interior of the domain. In fact, the $QR$-based
  normalization method we use is similar to that introduced
  in~\cite{trefethen2005}. The difficulties underlying eigenvalue
  search in the present integral-equation context are different from
  those found in the approach~\cite{trefethen2005}, however. Indeed,
  as discussed in Section~\ref{Section:simple scan}, in the former
  case a phenomenon related to the Riemann-Lebesgue lemma is at work:
  highly oscillatory integrands of unit norm can yield small
  integrals. In the latter case, in contrast, the root cause lies in
  the fact that linear combinations of a number $n$ of Bessel
  functions with coefficients of unit norm (say, in the mean square
  sense) can be selected which tend to zero as $n$ grows. (Notice
  that, in view of the $z\to 0$ asymptotics
  $J_n(z)\sim\mathcal{O}(z^n)$, this fact bears connections with a
  well known result concerning polynomial interpolation: linear
  combinations of $n$ monomials can be made to tend to zero rapidly as
  $n$ grows---for example, the monic Chebyshev polynomial of order $n$
  tends to zero exponentially fast as $n\to\infty$.)
\end{remark}

\subsection{Sign changing procedure for the minimum singular value. \label{flipping}}
Consideration of Figure~\ref{fig:comparison} clearly suggests that the function $\widetilde{\eta}_n$ is a continuous but non-smooth function of $\mu$. As is known~\cite{wright1992differential,trefethen2005}, however, a  sign-changing methodology suffices to produce singular values as smooth (indeed analytic) functions of $\mu$---so that high-order interpolation and root finding becomes possible. With reference to  Algorithm~\ref{halmos}, using approximate values of $\sigma_n(\mu_j)$ at points on the uniform mesh $\mu_j$, our algorithm relies on calculation of signed singular values and subsequent polynomial interpolation to approximate the zeros of the function $\widetilde{\eta}_n(\mu)$. The overall sign-changing/interpolation root-finding algorithm we use is essentially identical to that presented in~\cite[p. 488]{trefethen2005}. To obtain the approximate roots with prescribed error tolerance, further, we implement this procedure using nested uniform meshes around each approximate root found. 
\begin{remark}\label{remark:search_method}
The recent contribution~\cite{zhao2014robust} uses the Fredholm determinant to obtain an smooth function of $\mu$ that vanishes whenever $\mu$ corresponds to an eigenvalue, and it compares the efficiency of that solver to one based on consideration of singular values as a function of $\mu$ for which the singular values are merely piecewise smooth functions. The sign changing procedure described in this section, however, gives rise to smooth (analytic) dependence of the singular values as functions of $\mu$, and thereby eliminates the potential difficulties suggested in~\cite{zhao2014robust}. 
\end{remark}

\section{Eigenfunction evaluation \label{Section:eigenfunctions}}
After the eigenvalues are obtained, the corresponding eigenfunctions can be evaluated for both FC-based and graded-mesh solvers using the representation~\eqref{ansatz}. High-order approximate evaluation of this integral for points $x$ sufficiently far from the boundary $\Gamma$, for which the corresponding integral kernels are smooth, is performed using the quadrature~\eqref{eq:before_fc_repeated} in the case of a FC-based algorithm, and the combination of graded-mesh change of variables and a trapezoidal rule~\eqref{eq:trapezoidal_rule} in the case of a graded-mesh algorithm. Methods for the evaluation of the eigenfunction in case the point $x$ is on the boundary $\Gamma$ (more precisely, on the Neumann boundary portion $\Gamma_N$, since the eigenfunctions admit zero values on the Dirichlet portion $\Gamma_D$) have already been described in detail in sections~\ref{bound_dec}--\ref{Section:discretization_matrix}. Lastly, in case the point $x$ is close to the boundary $\Gamma$ the corresponding kernels exhibit a near-singularity. In this case an interpolation approach is used to evaluate the eigenfunction $u(x)$ at a point $x$: letting $x_0$ denote the point in $\Gamma$ that is closest to $x$ and letting $L$ denote a straight segment passing through the points $x$ and $x_0$,  the values of $u$ at a small set of points $x_j\in L$ ($j=0,1\dots$) that, except for $x_0$, are sufficiently far from the boundary $\Gamma$ are used to produce the value $u(x)$ by means of an interpolating polynomial. (Typically cubic or quartic polynomials were used to produce the images presented in this paper.) To reach a prescribed tolerance it may be necessary to use increasingly fine meshes $\{x_j\}$ for which some or all elements may be closer to $\Gamma$ than is required for accurate integration by means of the available boundary mesh. In such cases the Chebyshev boundary expansions that produce the solution can be oversampled (by means of zero padding of the corresponding cosine expansion) to a mesh that is sufficiently fine to produce sufficiently accurate integrals at each one of the points $x_j$---and the interpolation procedure then proceeds as indicated above.

\section{Numerical results \label{sec:numerical_results}}
This section presents results of numerical experiments which
demonstrate the accuracy, efficiency and high-order character of the
proposed eigensolver. In preparation for this discussion we note that
there are only a few Zaremba eigenproblems whose spectrum is known in
closed form: even for geometries such as a disc, which are separable
for both the pure Dirichlet or Neumann eigenproblem, no Zaremba
spectra for nontrivial selections of $\Gamma_D$ and $\Gamma_N$ have
been evaluated explicitly. We thus first demonstrate the performance
of our algorithms for a Zaremba problem for which the spectrum is
available: an isosceles right triangle. We next compute the first few
eigenvalues of the Zaremba problem in smooth domains (both convex and
non-convex) for which no spectra have previously been put
forth---either in closed form or otherwise. As reference solutions for
such problems we use results of computations we produced on the basis
of well-established and validated finite element
codes~\cite{hecht2007freefempp}. We emphasize that no attempt was made
to optimize these finite element computations beyond the use of mesh
adaption near singular points. In addition, for certain polygonal
domains with obtuse angles we compare our results with existing
validated numerical simulations~\cite{liu}, and we then demonstrate
the behavior of our algorithm in a number of challenging problems.

In all, our examples include:
\begin{enumerate}[1)]

\item A problem on a convex polygonal domain (with a Dirichlet-Neumann junction occurring at a vertex with angles of less than $\displaystyle \frac{\pi}{2}$;  (Section~\ref{sec:experiment2}).
\item An application of the FC-based solver to smooth, convex domain; (Section~\ref{sec:experiment3}).
\item An application of the FC-based solver to smooth non-convex domain; (Section~\ref{sec:experiment4}).
\item  A problem  on a polygonal domain with the Dirichlet-Neumann junction occurring at angle greater than $\displaystyle \frac{\pi}{2}$. In this case, we set up the experiments to compare with  corresponding theoretically-identical Laplace-Dirichlet eigenvalues of a symmetry-related domain; (Section~\ref{sec:experiment5}).
\item A comparison of the performance of the FC-based and graded-mesh algorithms described in Sections~\ref{sec:smooth_geometries} and~\ref{sec:corner_geometries} when both are applied to a smooth domain. (The superior performance of the FC-based algorithm dramatically improves the overall performance of the eigensolver); (Section~\ref{sec:experiment1}).
\item Applications concerning high-frequency eigenvalue problems (evaluating thousands of eigenvalues and eigenfunctions, and showing, in particular, that the proposed eigensolver can successfully capture the asymptotic distribution of Zaremba eigenvalues);  (Section~\ref{sec:experiment6}).
\item Generalization to multiply connected domains; (Section~\ref{sec:experiment7}).
\item Applications to pure Dirichlet and pure Neumann eigenvalue problems; (Section~\ref{sec:experiment8}).
\end{enumerate}

The numerical results presented in this paper were obtained on a single core of a 2.4 GHz Intel E5-2665 processor. All of the listed digits for eigenvalues produced by the various FC-based and graded-mesh eigensolvers are significant (with the last digit rounded to the nearest decimal), while in the values produced by the FEM methods a number of digits additional to the correct ones are presented (to avoid rounding the first or second decimal).

\subsection{Convex polygonal domains \label{sec:experiment2}}
In this section the performance of the graded-mesh algorithm on simple
polygonal domains is analyzed. Such domains provide instances of
geometries where true eigenvalues of problem~\eqref{modelA} can be
computed analytically using reflection techniques. In detail, we
consider the Zaremba problem on the isosceles triangle with corners
(0,0), (0,1) and (1,0). Neumann data is prescribed along one side of
unit length, and Dirichlet data is prescribed along the other two
sides. For this geometry the eigenvalues of~\eqref{modelA} are a
subset of the set of Neumann eigenvalues of a square with corners
$(-1,0), (1,0), (1,2)$ and $(-1,2)$.  More specifically, if we pick
Neumann eigenfunctions of the square which have the correct
symmetries, the corresponding eigenvalues are the same as those of the
Zaremba problem on the described triangle. We can explicitly compute
these eigenvalues on the square to be
\begin{equation}\label{eig_expr}
  \lambda_{k,\ell}=\displaystyle \frac{(2k+1)^2+(2\ell+1)^2}{4} \pi^2, \quad k,\ell=0,1,2,3...
\end{equation} 
The comparison of the approximate eigenvalues computed on the basis of
$328$ and $1200$-point boundary meshes and the exact eigenvalues is shown in
Table~\ref{table:triangle}; corresponding eigenfunctions are depicted
in Figure \ref{fig:triangle}.

\begin{table}[H]
  \centering
\begin{tabular}{c|c|r |r|r } 
	$k$ &	$\ell$ 	& $\lambda_{k,\ell}$ exact	 & $\lambda_{k,\ell}$  ($n=328$) & $\lambda_{k,\ell}$  ($n=1200$)	\\ \hline
	0 & 1  & 24.6740110027234 & 24.67401100  & 24.6740110027234 \\
	0 & 2  & 64.1524286070808  & 64.15242861 & 64.1524286070809	\\
	1  & 2  & 83.8916374092595 & 83.89163742 & 83.8916374092596	\\
	0 & 3 	& 123.3700550136170 & 123.3700550 & 123.3700550136170	\\ 
	1 & 3  & 143.1092638157957 & 143.1092638 & 143.1092638157957	\\
	2 & 3  & 182.5876814201531 & 182.5876814 & 182.5876814201531	\\
	0 & 4  & 202.3268902223318 & 202.3268902 & 202.3268902223319	\\
\end{tabular}
\caption{Eigenvalues $\lambda_{k,\ell}$ for the isosceles triangle considered in Section~\ref{sec:experiment2} produced by the proposed graded-mesh eigensolver with $n= 328$ and $n= 1200$ compared to results produced by the closed form  expression~\eqref{eig_expr}.}
\label{table:triangle} 
\end{table}

\begin{figure}
\centering
  \includegraphics[width=6in]{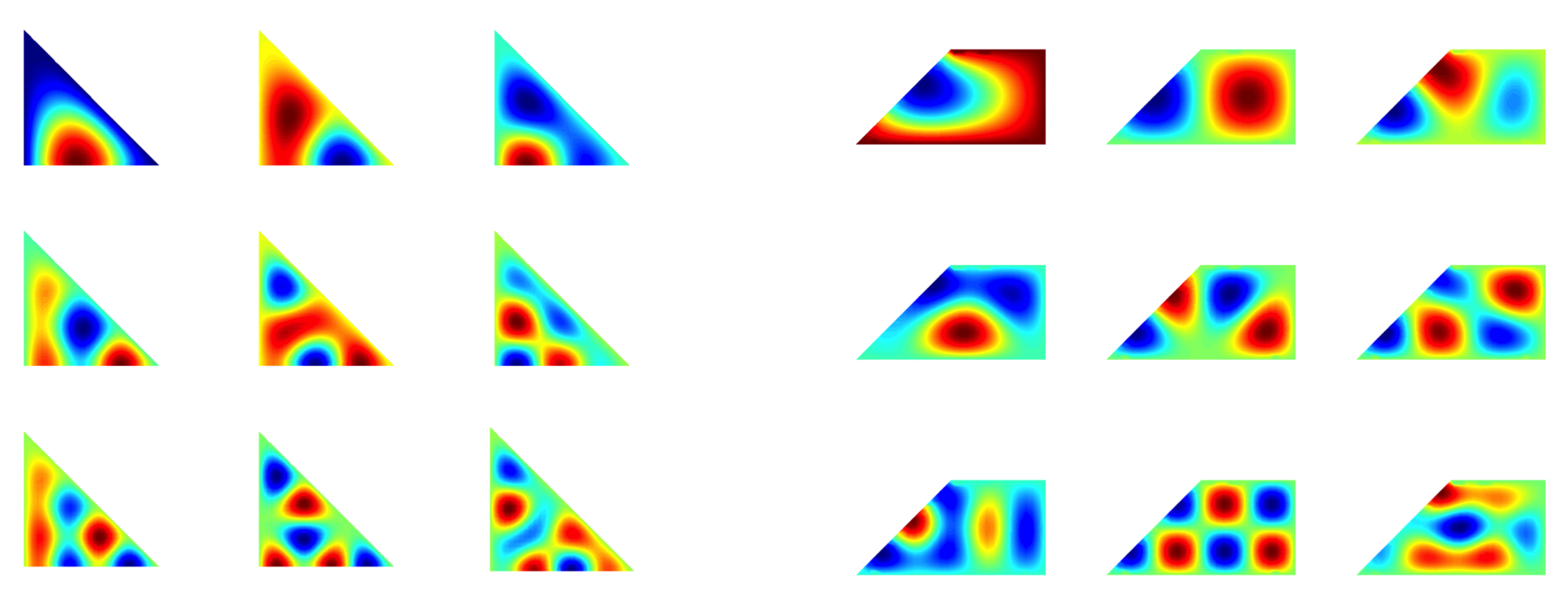}
 \caption{Left: Zaremba eigenfunctions for the triangular-domain eigenproblem considered in Section~\ref{sec:experiment2}. Right: Zaremba eigenfunctions for the trapezoid-shaped-domain eigenproblem considered in Section~\ref{sec:experiment5}.}
  \label{fig:triangle}
\end{figure}

\subsection{Convex smooth domains \label{sec:experiment3}}
In the case of smooth domains the Dirichlet-Neumann junction takes place at a vertex with interior angle equal to $\pi$, and, thus, the corresponding eigenfunction of~\eqref{modelA} is continuous but not twice continuously differentiable up to the boundary; see e.g.~\cite{wendland1979}. This fact gives rise to challenges for volumetric strategies; in particular, high-order conforming elements do not yield high-order accuracy for this problem. Our proposed boundary integral strategy coupled with the high-accuracy FC discretization of integral operators, in turn, efficiently provides high-order convergence and highly-accurate results.

These facts are illustrated in Tables~\ref{table:FEM} and~\ref{table:FC_method}, which present the first Zaremba eigenvalue for the unit disc as produced by the P1 and P2 FEM algorithm~\cite{hecht2007freefempp} and the FC-based eigensolvers. Clearly  the convergence resulting from the FEM methods is slow: Table~\ref{table:FEM} shows that, even using a mesh containing over 10,000-triangles, the FEM methods produce results with no more than 2 digits of accuracy. The FC results displayed in Table \ref{table:FC_method}, in turn, demonstrate that the FC solver produces eigenvalues with 10 digits of accuracy using a discretization containing a mere 512 mesh points.

\begin{table}[H]
\centering
\begin{tabular}{c|c|c } 
$N_t$ &	P1	& 	P2 	\\ \hline
636&1.59&1.56\\
2538&1.57&1.56\\
10120&1.56&1.55\\
39962&1.55&1.55\\
\end{tabular}
\caption{Convergence of first Zaremba eigenvalue on the disc.  P1 and P2  FEM approaches. Here $N_t$ is the number of triangles in the mesh.}
\label{table:FEM} 
\end{table}
 \begin{table}[H]
\centering
\begin{tabular}{ l |c |  c |c | r  }
$n$ & 64 & 128 & 256 & 512\\ \hline
$\lambda_1$ & 1.548549 & 1.54854933 & 1.5485493331 & 1.548549333189 \\
\end{tabular}
\caption{Convergence of the FC-based eigensolver:  first Dirichlet-Neumann Laplace eigenvalue in the unit disc.}
\label{table:FC_method} 
\end{table}

For reference, Table~\ref{Table:discmany} presents corresponding results produced by P1, P2 and P1 Non-Conforming (Crouziex-Raviart) FEM methods as well as the proposed FC-algorithm for the first 10 Zaremba eigenvalues on the disc. Once again the convergence of the FEM algorithms is slow, and, using tens of thousands of unknowns, yield no more than 3 digits of accuracy. The corresponding 10 eigenvalues produced by the FC method, in turn, do contain at least a full 13 digits of accuracy.

\begin{table}[h!]
\centering
\begin{tabular}{c|c|c|c }
P1&		P1 NC		&P2&  FC eigensolver\\ \hline
1.55	&	1.54	&	1.55	&	1.548549333189\\
6.68	&	6.64	&	6.68	&	6.668097160848\\
8.66	&	8.66	&	8.66	&	8.662779904509\\
14.82	&	14.74	&	14.80	&	14.782583814100\\
17.86	&	17.83	&	17.85	&	17.848357621645\\
21.21	&	21.20	&	21.20	&	21.204559421807\\
25.83	&	25.73	&	25.81	&	25.781212572974\\
29.65	&	29.56	&	29.63	&	29.605375911651\\
35.93	&	35.90	&	35.92	&	35.914231714109\\
37.80	&	37.74	&	37.78	&	37.767236907914\\
\end{tabular}
\caption{The first 10 Zaremba eigenvalues on the unit disc. The P1 conforming and P1 non-conforming computations are on a mesh of 40144 triangles (3 digit accuracy). The P2 conforming FEM computations are on 10136 triangles (3 digit accuracy). In contrast, 512 points suffice for the FC eigensolver (Section~\ref{sec:smooth_geometries}) to produce the eigenvalues with an accuracy of 13 digits).}
\label{Table:discmany}
\end{table}

\subsection{Smooth, non-convex domains \label{sec:experiment4}}

In this experiment we consider a non-convex domain with smooth boundary parametrized by 
\begin{eqnarray} x_1=\cos(t)+0.65 \cos(2t)-0.65\quad\mbox{ and }\quad x_2=1.5 \sin(t), \end{eqnarray}
with the Neumann and Dirichlet boundary portion $\Gamma_N$ and $\Gamma_D$ corresponding to  $ t \in [\pi/2; 3\pi/2]$ and its complement, respectively. No exact solution for this problem is available. 
We compare the performance of our FC-based algorithm with the performance of three finite element methods: a P1 conforming method, a P1 non-conforming (Crouziex-Raviart) method, and a P2 conforming method. The convergence of the finite element methods is once again slow: we present  FEM results with three significant digits which require tens of thousands of unknowns. In contrast, the FC-based eigensolver yields 13 digits of accuracy on the basis of a mere 512 points boundary discretization.  These results are detailed in Table \ref{Table:kite}.
\begin{table}[h!]
\centering
\begin{tabular}{ c|c|c|c }\\
Crouzeix-Raviart & P1 conforming&P2 conforming& FC-based eigensolver \\ \hline
2.49	&	2.49	&	2.49	&	2.494957693616\\
6.24	&	6.26	&	6.25	&	6.253748349225\\
8.03	&	8.04	&	8.04	&	8.042440637044\\
12.03	&	12.08	&	12.06	&	12.053365383455\\
13.38	&	13.43	&	13.42	&	13.406406452033\\
18.04	&	18.06	&	18.05	&	18.047848229702\\
19.08	&	19.20	&	19.17	&	19.137393493839\\
\end{tabular}
\caption{Numerical experiments for the kite-shaped domain. The  P1 (both conforming and non-conforming) FEM methods  use 40144-triangle meshes, whereas the P2 method uses a 5156-triangle mesh. The FC-based eigensolver uses 512 boundary points.}
\label{Table:kite}
\end{table}
\begin{figure}[h!]
\centering
  \includegraphics[width=4in]{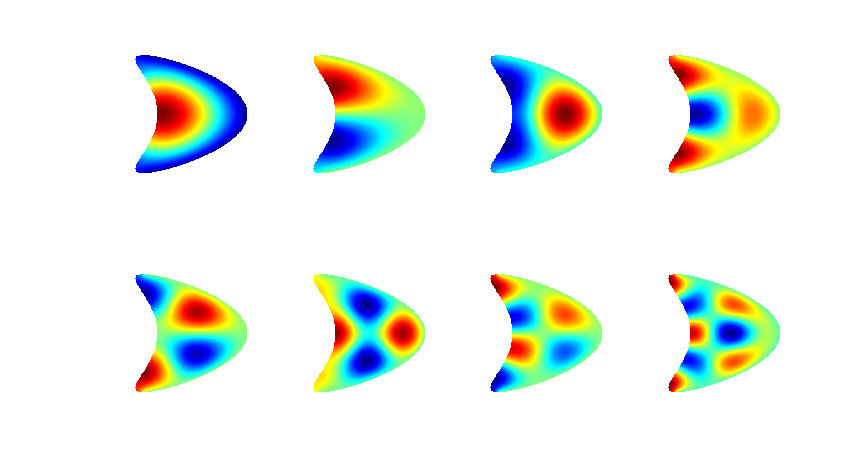}
 \caption{Zaremba eigenfunctions for the kite-shaped-domain eigenproblem described in Section~\ref{sec:experiment4} }
  \label{fig:KITE}
\end{figure}

In Table~\ref{Table:CompTimesFC} we list the computational times for the FC-based solver to compute first 12 eigenvalues for the geometries considered in Sections~\ref{sec:experiment2} and ~\ref{sec:experiment4}. All times are given on a per-eigenvalue basis.

\begin{table}[h!]
\centering
\begin{tabular}{c | c |  c }\\
Domain	&	Time (err. $10^{-5}$) 	&	Time (err. $10^{-10}$) \\ \hline
Disc	&	0.09 s	&	1.13 s  \\
Kite	&	0.18 s	&	1.88 s	 \\
\end{tabular}
\caption{FC-based eigensolver. Computational times per-eigenvalue for the first 12 eigenvalues }
\label{Table:CompTimesFC}
\end{table}

\subsection{Polygonal domains with obtuse Dirichlet-Neumann junctions \label{sec:experiment5}}

The L-shaped domain provides an important test case. In reference~\cite{liu} a set of validated numerical experiments is presented for the Dirichlet eigenvalue problem on an L-shaped domain (a square of side length two with a unit square removed). These numerical results were produced by means of finite element discretizations. For the first Dirichlet eigenvalue, a {\it provable} interval $[9.5585,9.6699]$ which brackets the true eigenvalue is provided.

\begin{table}[h!]
\centering
\begin{tabular}{c | c | c | c }\\
$i$	&	Lower bound	&	Upper bound	&	$\lambda_i$ (graded-mesh eigensolver) \\ \hline
1	&	9.55	&	9.66	&	9.639723844021955 \\
3	&	19.32	&	19.78	&	19.739208802178748 \\
5	&	30.86	&	32.05	&	31.912635957137709 \\
\end{tabular}
\caption{Eigenvalues corresponding to the symmetric eigenfunctions for the L-shaped domain. Comparison with table 5.5 in \cite{liu}. Eigenvalues produced by means of the graded-mesh eigensolver are computed with at least 13 digits of accuracy (by convergence analysis).}
\label{Table:Lshape}
\end{table}

Using symmetry arguments it can be easily seen that some of the Zaremba eigenvalues for the trapezoid that results by cutting the L-shaped domain along a symmetry line coincide with certain Dirichlet eigenvalues on the L-shaped domain.  The graded-mesh algorithm introduced in this paper produces the approximation $9.639723844021955$ for the first eigenvalue---clearly within the guaranteed interval---and several other eigenvalues are computed without difficulty, see Figure~\ref{fig:triangle}. Table~\ref{Table:Lshape} displays the eigenvalue bounds resulting from use of the FEM algorithm from reference~\cite[Table 5.5]{liu} as well as those produced by means of the graded-mesh algorithm presented in this paper. Figure~\ref{fig:triangle} (right) presents depictions of several Zaremba eigenfunctions on the trapezoid mentioned above.
In Table~\ref{Table:CompTimes} we list the computational times for the graded-mesh solver to compute first 12 eigenvalues for the geometries considered in Sections~\ref{sec:experiment3} and~\ref{sec:experiment5}. All times are given on a per-eigenvalue basis. 

\begin{table}[h!]
\centering
\begin{tabular}{c | c |  c }\\
Domain	&	Time (err. $10^{-5}$) 	&	Time (err. $10^{-10}$) \\ \hline
Triangle	&	0.07 s 	&	1.19 s	\\
Trapezoid 	&	0.18 s 	&	1.89 s \\
\end{tabular}
\caption{Graded-mesh eigensolver. Computational times per-eigenvalue for the first 12 eigenvalues.}
\label{Table:CompTimes}
\end{table}

\subsection{Comparison of FC-based and graded-mesh approaches \label{sec:experiment1}}

Sections \ref{sec:smooth_geometries} and \ref{sec:corner_geometries}   describe FC-based and graded-mesh eigensolvers for high-order evaluation of Zaremba eigenvalues on smooth and Lipschitz geometries, respectively. As indicated in Remark~\ref{rem_FC_graded}, however, the graded-mesh algorithm can also be applied  to smooth geometries.  Figure~\ref{fig:fc_vs_ck} compares the convergence history for both of these algorithms as they are used to obtain the Zaremba eigenvalue $\lambda_{18} = 73.1661817902$ for the unit disc (where Dirichlet and Neumann boundary conditions are prescribed on the upper and lower halves of the disc boundary). This figure demonstrates a general fact: for smooth geometries the FC-based approach significantly outperforms the (more generally applicable) graded-mesh algorithm. The somewhat slower convergence of the graded-mesh solver relates, in part, to the relatively large value $\alpha=1$ associated with the $180^\circ$ angle that occurs at Dirichlet-Neumann junctions on smooth curves; cf. Remark~\ref{remark:cancellations_corner}.

\begin{figure}[h!]
\centering
\includegraphics[width=5.5in]{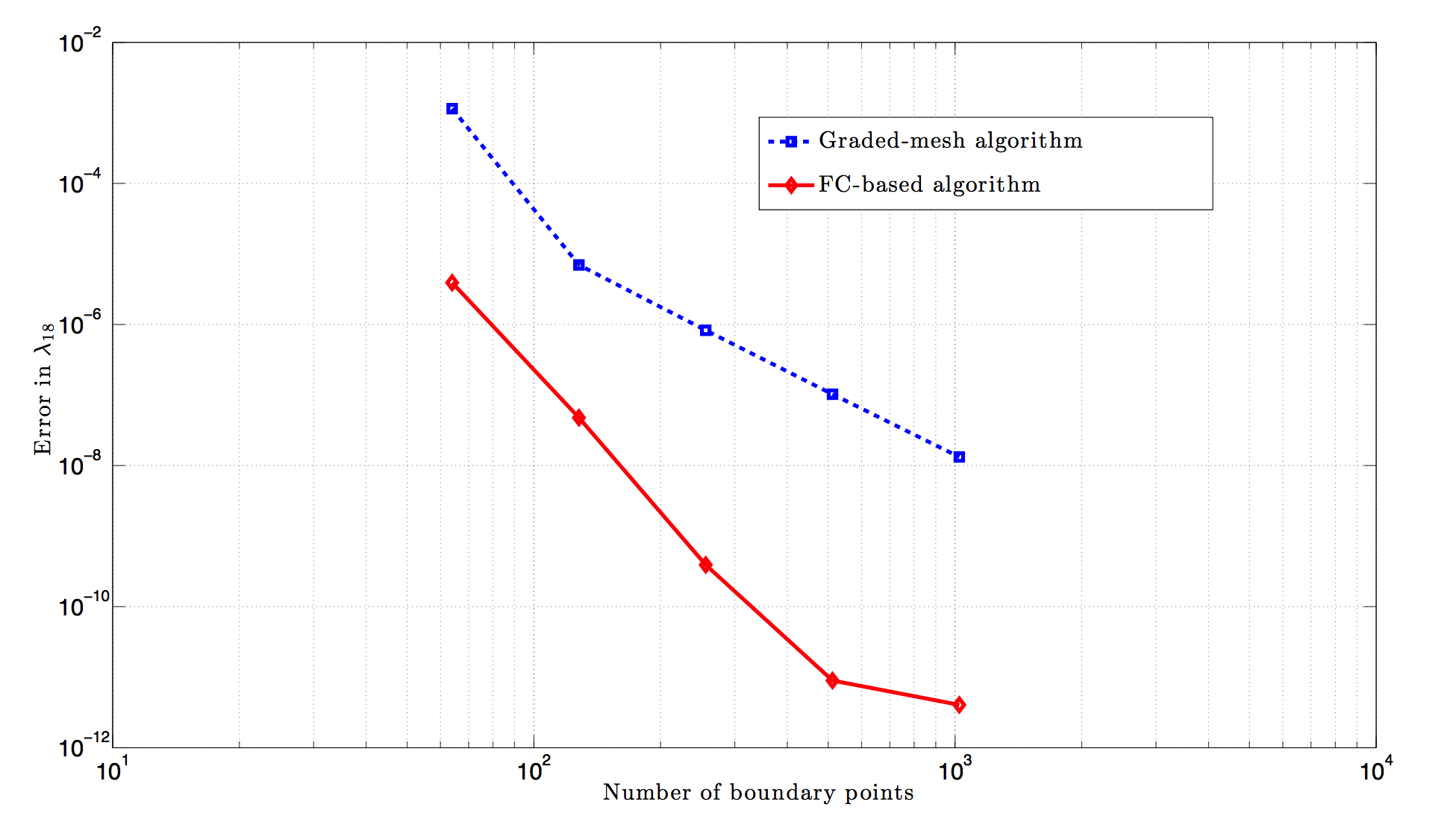}
 \caption{Sample convergence history resulting from the FC-based eigensolver and the graded-mesh eigensolver for the eighteenth Zaremba eigenvalue $\lambda_{18}$ discussed in Section~\ref{sec:experiment5}. The computational times required for evaluation each one of the eigenvalue approximations by means of the FC solver and the graded-mesh solver are as follows. FC-solver times: 0.23s, 0.67s, 2.82s, 19.19s, 119.7s. Graded-mesh solver: 0.09s, 0.60s, 1.93s, 16.11s, 112.90s. We note, for example,  that an error of $10^{-7}$ results from the FC-solver in this case in a computational time of 0.67 seconds; for the same accuracy, the computing time required by the graded mesh solver is 16.11 seconds.}
  \label{fig:fc_vs_ck}
\end{figure}

\subsection{High-frequency wave numbers \label{sec:experiment6}}

\begin{figure}[h!]
\centering
\includegraphics[width=5in]{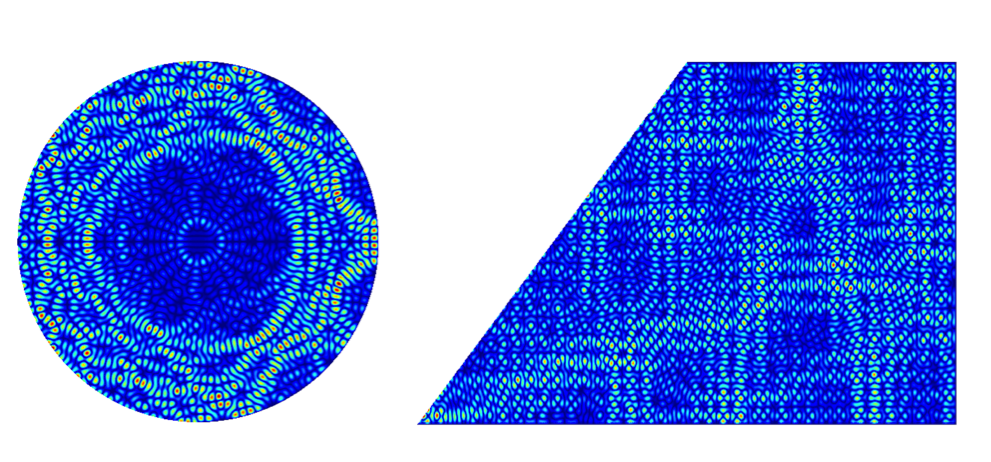}
 \caption{High frequency eigenfunctions mentioned in Section~\ref{sec:experiment6}}
  \label{fig:HF_PLOT}
\end{figure}

The high-order convergence of the algorithms presented in this paper
enables evaluation of eigenvalues and eigenfunctions in very wide
frequency ranges.  For example, we have used our solver to produce the
first 3668 eigenvalues and eigenfunctions for the eigenproblem
mentioned in Section~\ref{sec:experiment5} with a full 13 digits of
accuracy (the eigenvalues are depicted in Figure~\ref{fig:weyl}
(left)). The single-core computational time required for evaluation of
the first 9 eigenvalues the was 17 seconds, while for the last 9
eigenvalues (that correspond to higher values of $\lambda$, and,
therefore, finer discretization meshes required for a given accuracy)
the computational times was 189 minutes.  In another example, Figure
\ref{fig:HF_PLOT} shows an eigenfunction for a unit disc corresponding
to the eigenvalue $\lambda=10005.97294969$ (left) and an eigenfunction
for a trapezoid (that also corresponds to symmetric Laplace-Dirichlet
eigenfunction for L-shaped domain (cf.  Section~\ref{sec:experiment5})
corresponding to the eigenvalue $\lambda=40013.2312203$ (right).


Our next experiment concerns the number $N(x)$ of Dirichlet-Neumann eigenvalues $\lambda$ satisfying $0 < \lambda \leq x$. For pure Dirichlet eigenvalues $N(x)$ satisfies the Weyl asymptotics~\cite{weyl1912asymptotische,weyl1911asymptotische} 
\begin{equation}\label{weyl_asymptotics}
\lim_{x \to \infty} \frac{N(x)}{x} = (2 \pi)^{-d} \omega_d A(\Omega),
\end{equation}
(see also~\cite[p. 322-323]{strauss2008partial}), where  $\omega_d$ is the volume of the unit ball in $d$ dimensions ($\omega_d = \pi$ for $d=2$) and where $A(\Omega)$ is the volume of the domain $\Omega\subset \mathbb{R}^d$. Figure~\ref{fig:weyl} depicts the ratio $\displaystyle \frac{N(x)}{x}$ for the Dirichlet-Neumann eigenvalues for the geometries considered earlier in this Section: triangle (Section ~\ref{sec:experiment2}), unit disc (Section~\ref{sec:experiment3}), kite (Section~\ref{sec:experiment4}) and trapezoid (Section~\ref{sec:experiment5}). These results suggest that a similar limit may exist for the Zaremba eigenvalue problem.

\begin{figure}[h!]
\centering
\includegraphics[width=6.5in]{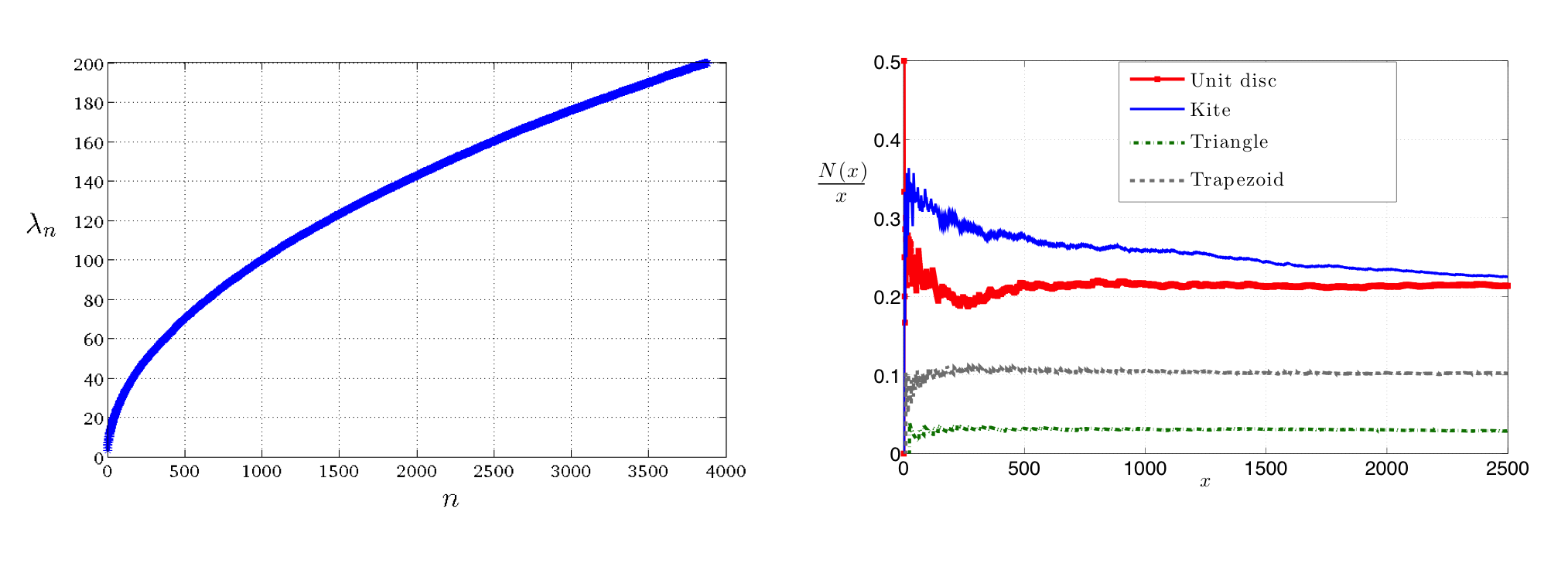}
  \caption{Left: First 3668 eigenvalues for the trapezoidal domain eigenproblem considered in Section~\ref{sec:experiment5}. Right: Ratio $\frac{N(x)}{x}$  (cf. eq.~\eqref{weyl_asymptotics}) for various domains.}
  \label{fig:weyl}
\end{figure}

\subsection{Multiply connected domains \label{sec:experiment7}}
This section presents results produced by a generalization of our eigensolvers that can be applied to multiply connected domains. As is well known, application of integral eigensolvers to multiply connected domains can give rise to spurious resonances~\cite{chen2001boundary,chen2003spurious}---which arise from eigenvalues of the domain interior to the inner boundary. We have found, however, that by enforcing an additional condition based on use of interior points in a manner related to that considered in Section~\ref{Section:EVP} (which ensures that the function $u$ in equation~\eqref{ansatz} vanishes in the bounded components of the complement of $\Omega$) a function $\widehat{\eta}_n(\mu)$ is obtained that is equal to zero only at the true eigenvalues of the multiply connected domain. The detailed description of the algorithm will be presented elsewhere; here we provide a preliminary numerical demonstration of the new methodology. The domain under consideration is a polygon determined by the set of exterior vertices $(0,0), (0,3), (2,4), (3,2), (3,0)$ and the set of interior vertices $(1,1), (1,2), (2,3)$ with Dirichlet boundary conditions on all sides. Figure~\ref{fig:house} depicts true and spurious eigenfunction for this domain corresponding to the eigenvalues $\lambda = 76.619031$ and $\lambda = 77.663162$, respectively. As can be seen in Figure~\ref{fig:multiply_connected}, the procedure effectively screens out the spurious eigenvalue $\lambda = 77.663162$.


\begin{figure}[h!]
\centering
\includegraphics[width=5in]{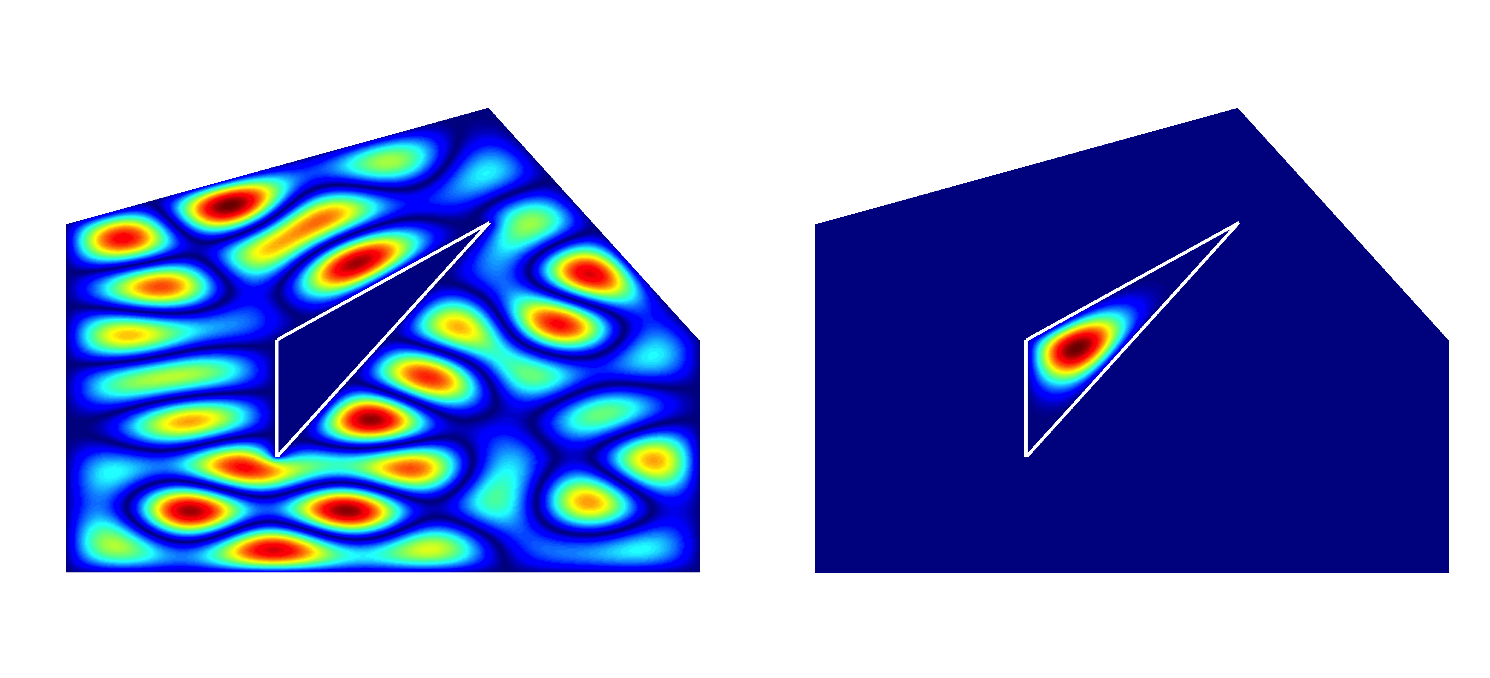}
  \caption{True and spurious eigenfunction for multiply connected domain. Left: true eigenfunction corresponding to $\lambda = 76.619031$. Right: spurious eigenfunction corresponding to $\lambda = 77.663162$ }
  \label{fig:house}
\end{figure}
\begin{figure}[h!]
\centering
\includegraphics[width=4.5in]{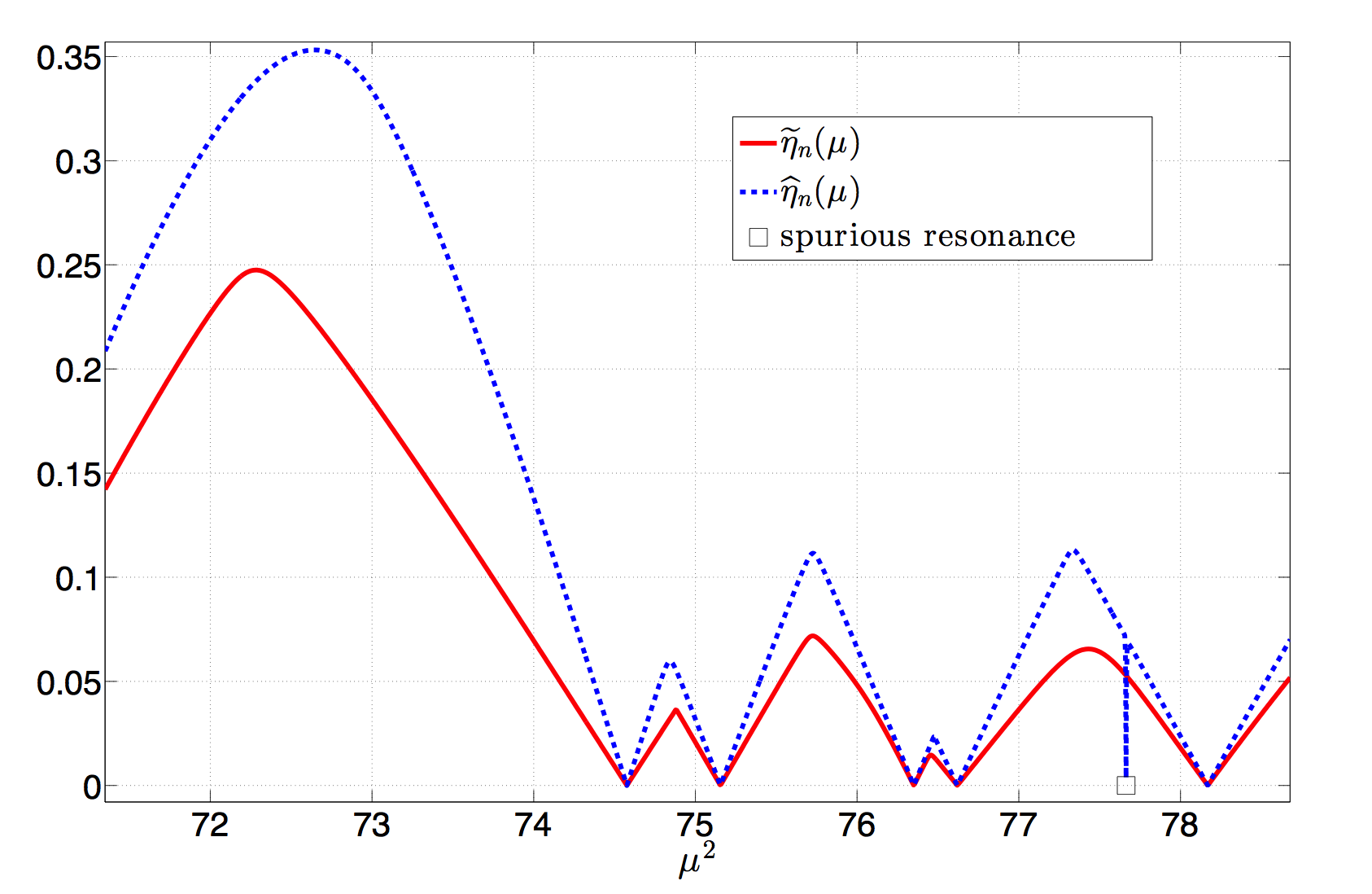}
  \caption{Filtering of spurious eigenvalue. Dashed curve: function $\widetilde{\eta}_n(\mu)$; solid curve: $\widehat{\eta}_n(\mu)$ }
  \label{fig:multiply_connected}
\end{figure}


\subsection{Pure Dirichlet and Pure Neumann eigenfunctions \label{sec:experiment8}}
As mentioned in the introduction, the methods described in this paper can be applied to a variety of eigenvalue problems (see Remark~\ref{first_kind} for a discussion in these regards). In particular, Laplace eigenfunctions for pure Dirichlet or pure Neumann boundary conditions can be computed using the proposed eigensolver: both problems can be treated as particular cases of the more general Zaremba problem (for which $\Gamma_D = \emptyset$ or $\Gamma_N = \emptyset$, respectively). Sample eigenfunctions produced by our methods under pure Dirichlet and pure Neumann boundary conditions are presented in Figure~\ref{fig:lshape_dirichlet}. 
\begin{figure}[h!]
\centering
\includegraphics[width=3in]{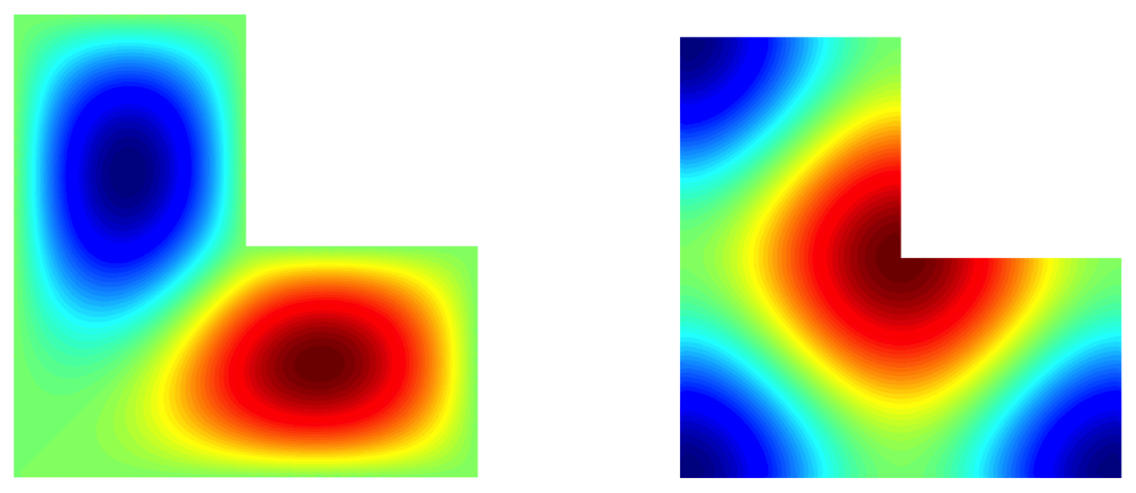}
\caption{Eigenfunctions for L-shaped domain with Dirichlet (left) and Neumann (right) boundary conditions}
  \label{fig:lshape_dirichlet}
\end{figure}

\section{Conclusions}
This paper introduces a novel integral-equation based strategy for solution of Zaremba eigenproblems. By precisely accounting for the singularities of the boundary densities and kernels, the relevant boundary integral operators are discretized with accuracies of very high order. Methods are presented for smooth domains (based on Fourier Continuation techniques) and for Lipschitz domains (based on use of graded meshes). A stabilization technique used as part of our zero-singular-value search algorithm yields a robust non-local eigenvalue-search method. The resulting solvers allow for highly accurate and efficient approximation of eigenvalues and eigenfunctions, even for cases that involve strongly singular eigenfunctions and/or very high frequencies.

\section*{Acknowledgments}
EA and OB gratefully acknowledge support from AFOSR and NSF. NN gratefully acknowledges support from NSERC and the Canada Research Chairs foundation.
\appendix
\section{Appendix: The Fourier Continuation method (FC)\label{sec:FC}}

Given $N$ point values $f(x_i)$ ($x_i=\frac{i \pi}{N-1}$,
$i=0,\dots,N-1$) of a smooth function $f(x)$ defined in the interval
$[0,\pi]$, the Fourier Continuation algorithm produces rapidly
convergent periodic approximations $f^c$ of $f$ to an interval of
length larger than $\pi$. In view of the closed-form
integrals~\eqref{symms_operator}-\eqref{eq:Am_formula} used in
Section~\ref{sec:high_order_quadratures}, which lie at the basis of
our FC-based quadrature method, in the context of the present paper
the needed periodicity length is $2\pi$---so that the Fourier
continuation of the function $f$ takes the form
\begin{equation}
\label{eq:FC_description}
 f^c(x)=\sum _{k=-F}^{F} a_k e^{i k x}
\end{equation}
for some value of $F$. (The form~\eqref{eq:FC_description} applies to
expansions with an odd number $2F+1 $ of terms, but obvious
alternative forms may be used to include expansions containing an even
number of terms.) In this paper we use the ``blending-to-zero"
version of the algorithm, which was introduced in~\cite{albin2011},
together with small additional adjustments to enable use of the long
continuation intervals required in the present paper. For additional
details, including convergence studies of FC approximations, we refer
to \cite{bruno2009_1,bruno2009_2,albin2011}.

The extended periodicity interval is used in the FC method to
eliminate discontinuities that arise in a period-$\pi$ extension of
the function $f$, and thus, to eliminate the difficulties arising from
the Gibbs phenomenon. The FC representation \eqref{eq:FC_description}
is based on use of a preliminary discrete extension of $f$ to the
interval $[\pi-L,L]$ ($L>\pi$) which contains $[0,\pi]$ in its
interior. This discrete extension is obtained by appending to the
original $N$ function values an additional $C > 0$ function values
that provide a smooth transition from $f_{N-1}$ to $0$ in the interval
$[\pi, L]$, as well as $C$ function values that provide a smooth
transition from $f_0$ to zero in the interval $[\pi-L,0]$. Here
$\displaystyle L=\pi (N+C)/(N-1)$ with $C$ small enough so that $L <
3\pi/2$.

To obtain the function values in the extension domains $[\pi-L,0]$ and
$[\pi,L]$ we use a certain FC(Gram) algorithm~\cite{bruno2009_1} which
is briefly described in what follows. The FC(Gram) method constructs,
at first, a polynomial approximant to $f$ in each one of the intervals
$[x_0,x_{d-1}]$ and $[x_{N-d},x_{N-1}]$ (for some small integer number
$d$ independent of $N$) on the basis of the given function values at
the discretization points $x_0,x_1,\dots,x_{d-1}$ and
$x_{N-d},x_{N-d+1},\dots,x_{N-1}$, respectively; see
Figure~\ref{fig:FC}. Following~\cite{bruno2009_1}, in this paper these
interpolants are obtained as projections onto a certain basis of
orthogonal polynomials: the Gram polynomial basis of order $m$. The
FC(Gram) algorithm then utilizes a precomputed smooth function for
each member of the Gram basis which smoothly blends the basis
polynomial to the zero function over the distance $L -\pi$;
see~\cite{bruno2009_1,bruno2009_2,albin2011} for details.

In view of the large continuation intervals required in this paper,
the function values on the interval $[\pi-L,L]$ produced as indicated
above are subsequently padded by an appropriate number of zero values
to produce values of a $2\pi$-periodic smooth function (see
Figure~\ref{fig:FC}).  The algorithm is completed via an application
of the Fast Fourier Transform (FFT) to the $2\pi$ periodic extended
discrete function---to produce the coefficients $a_k$ of the Fourier
continuation $f^c$ shown in \eqref{eq:FC_description}. Throughout this
paper we have used the parameter values $C=27$, $d=6$ and $m=5$.
\begin{figure}
  \centering
  \includegraphics[width=6in]{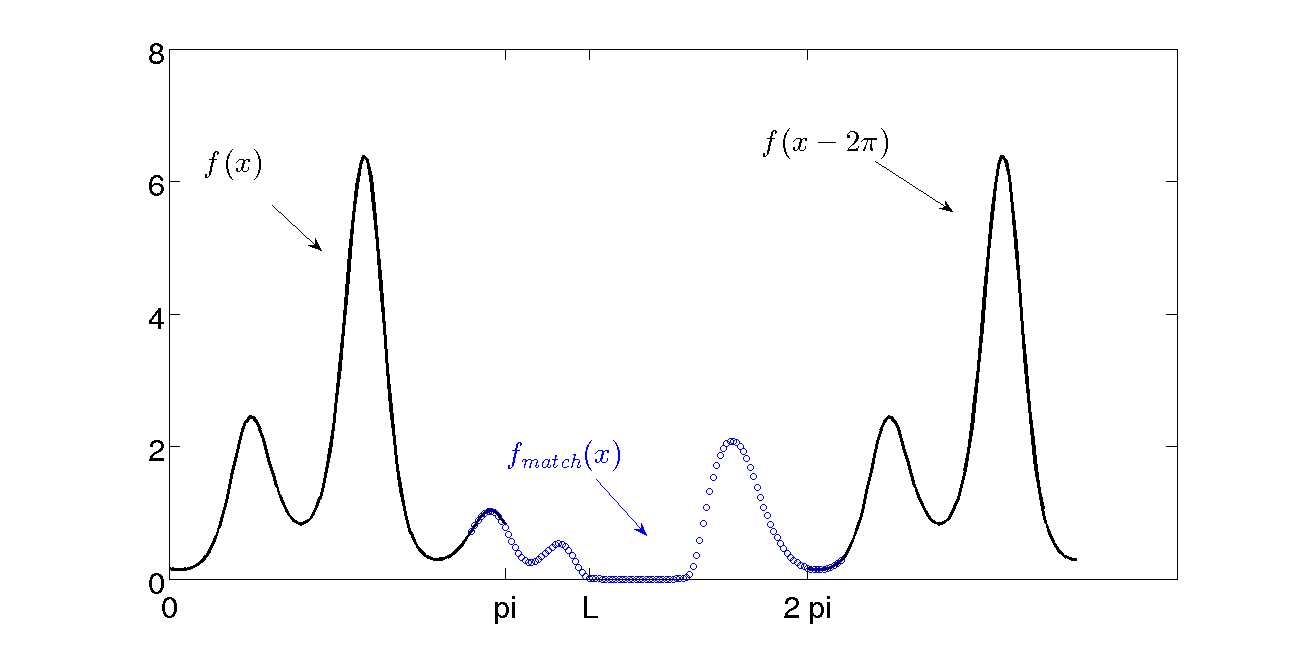}
  \caption{Demonstration of the blending-to-zero FC algorithm}
  \label{fig:FC}
\end{figure}

\bibliographystyle{plain}
\bibliography{Library}

\end{document}